\documentclass[12pt]{amsart}
\usepackage{amssymb,latexsym,amsfonts,verbatim,amscd,ifthen}

\newcommand{\C}{{\mathbb C}}
\newcommand{\R}{{\mathbb R}}
\newcommand{\sign}{\textrm{sign}}

\newcommand{\Po}{{\mathcal P}}

\numberwithin{equation}{section}

\theoremstyle{plain}

\newtheorem{theorem}[equation]{Theorem}
\newtheorem{lemma}[equation]{Lemma}
\newtheorem{proposition}[equation]{Proposition}
\newtheorem{corollary}[equation]{Corollary}
\newtheorem{problem}[equation]{Problem}

\newtheorem*{definition*}{Definition}
\newtheorem*{theoremA}{Theorem A}
\newtheorem*{theoremB}{Theorem B}
\newtheorem*{theoremC}{Theorem C}
\newtheorem*{theoremD}{Theorem D}

\theoremstyle{definition}

\newtheorem{remark}[equation]{Remark}
\newtheorem{definition}[equation]{Definition}

\begin{document}
\title
[Algebraic curves and operator tuples]
{Spectral algebraic curves  and decomposable operator tuples}
\author[M. I. Stessin]{M. I. Stessin}
\address{Department of Mathematics and Statistics \\
University at Albany, SUNY \\
Albany, NY 12222, USA}
\email{mstessin@albany.edu}

\author[A. B. Tchernev]{A. B. Tchernev}
\address{Department of Mathematics and Statistics \\
University at Albany, SUNY \\
Albany, NY 12222, USA}
\email{atchernev@albany.edu}


\begin{abstract}
Joint spectra of tuples of operators are subsets in complex projective space.
The corresponding tuple of operators can be viewed as an infinite
dimensional analog of a determinantal representation of the
joint spectrum.
We investigate the relationship between the geometry of the spectrum
and the properties of the operators in the tuple
when these operators are self-adjoint. In the case when the spectrum
contains an algebraic curve passing through an isolated spectral point of
one of the operators we give necessary and sufficient
geometric conditions for the
operators in the tuple to have a common reducing subspace.
We also address spectral continuity and obtain a norm estimate for the
commutant of a pair of self-adjoint matrices in terms of the Hausdorff
distance of their joint spectrum to a family of lines.
\end{abstract}

\maketitle
% Create a new 1st level heading
\section{Introduction}

When $A_1,\dots,A_n$ are $N\times N$ complex matrices,
the determinant
\begin{equation}\label{determinant}
{\mathcal S}(x_1,\dots,x_n)=\mbox{det}(x_1A_1+\dots+x_nA_n)
\end{equation}
is a homogeneous polynomial of degree $N$ in the variables
$x_1,\dots,x_n$, and zeros of this polynomial
determine a hypersurface in the projective
space $\C{\mathbb P}^{n-1}$. Conversely, given a
hypersurface $\Gamma$ of degree $N$
in $\C{\mathbb P}^{n-1}$,
if there are $N\times N$ matrices $A_1,\dots,A_n$ such that
$$
\Gamma=\{ \det(x_1A_1+ \dots +x_nA_n)=0 \}
$$
then the tuple $(A_1,\dots, A_n)$
is called a determinantal representation of $\Gamma$.
Of course, if a determinantal
representation exists, then for every pair of
invertible $N\times N$ matrices ${\mathcal M}$ and
${\mathcal N}$, the tuple of matrices
$
({\mathcal M}A_1{\mathcal N},\dots,
{\mathcal M}A_n{\mathcal N})
$
determines the same projective hypersurface, and these two
determinantal representations are called
equivalent.

A classical avenue of research in algebraic geometry
with a long history, see e.g. \cite{Ca, D, Do, KV, V1, R},
is determining
when a given hypersurface admits
a determinantal representation, and classifying all
such representations. If the hypersurface is real
(that is, the corresponding polynomial has real
coefficients), when is there a determinantal
representation with self-adjoint matrices? %$A_1,\dots,A_n$?
When is there  a determinantal representation with
triangular, or decomposable matrices, etc.?
We would like to mention specifically self-adjoint
determinantal representations of real curves, and
decomposable representations of reducible curves, since they
are close to the subject of this paper. The former produce
hyperbolic polynomials and are important in relation to the
Lax conjecture, cf. \cite{L,LPR,HMV}. The latter have
special meaning in operator theory. These representations
were considered in \cite {KV}.

The point of view from operator theory leads us to a second
natural avenue of research, that seems to have attracted
less attention in algebraic geometry: given that a
hypersurface has a determinantal
representation (or self-adjoint representation), what
does the geometry of the hypersurface
say about mutual relationships between the matrices
$A_1,\dots,A_n$?
In this direction
we would like to mention the result
of Motzkin and Taussky \cite{MT}, which states that %the following.
a real curve in $\C{\mathbb P}^2$ with a self-adjoint
determinantal representation satisfies the condition:
the matrices of the corresponding tuple
commute  if and only if this
curve is a union of projective lines (in \cite{MT} the
result is stated in equivalent but different terms).

In 2009 R. Yang \cite{Y1} started an investigation of what
can be called infinite dimensional determinantal
representations. Since well-known definitions of
spectra of a tuple of operators such as Taylor spectrum,
cf. \cite{Ta,EP}, are defined for commuting tuples, Yang was
looking for a good definition of joint spectrum  for
non-commuting operators  and introduced the notion of
joint spectrum of a tuple $(A_1,\dots, A_n)$
of operators acting on a Hilbert space $H$:

\begin{definition*}
The \emph{joint spectrum} $\sigma(A_1,\dots,A_n)$
of $A_1,\dots,A_n$ consists of all
$(x_1,\dots,x_n)\in \C^n$ such that
the operator $x_1A_1+ \dots +x_nA_n$ is not invertible on $H$.
If $A_n=I$, the identity operator, the
\emph{proper part} of the
joint spectrum of $A_1,\dots,A_{n-1}$ is defined as
$
\sigma_p(A_1,\dots,A_{n-1})=
\sigma(A_1,\dots,A_{n-1},I) \cap \{x_n = -1\}
$
\end{definition*}

Joint spectra were further investigated in
\cite{CY,BCY,SYZ}. It is easily seen that if
$(x_1,\dots,x_n)\in \sigma (A_1,\dots,A_n)$, then
the whole complex line $\{ (cx_1,\dots,cx_n): \ c\in \C \}$
lies in $\sigma (A_1,\dots,A_n)$, and, therefore,
$\sigma(A_1,\dots,A_n)$  determines a set in $\C{\mathbb P}^{n-1}$.
By analogy with the finite dimensional case, given a set
$\Gamma$ in $\C{\mathbb P}^{n-1}$, if there are operators
$A_1,\dots,A_n$ acting on a Hilbert space $H$ such that
\[
\Gamma=
\{ [x_1: \ \ldots \ : x_n]\in \C{\mathbb P}^{n-1} \mid
x_1A_1+\dots+x_nA_n \ \text{is not invertible} \},
\]
then it is natural to call the tuple $(A_1,\dots,A_n)$ a
\emph{spectral representation} of $\Gamma$.
The main difference compared to the classical matrix
case is that this set is not necessarily an analytic set.
For example, if $A_1$ and $A_2$ are compact and of
infinite rank, and $A_3$ is invertible, the whole line
$\{ [x_1 : x_2 : 0]\}$ in $\C{\mathbb P}^2$ is contained in the
joint spectrum and the spectrum is not an analytic set
near each point of this line. It was shown in  \cite{SYZ}
that if $A_1,\dots,A_{n-1}$ are compact and  $A_n$ is invertible
(and, therefore, can be  considered to be identity)
the part of the joint spectrum that lies in the chart
$\{ x_n\neq 0\}$ is an analytic set. When the operators
$A_1,\dots,A_{n-1}$ are trace class, that part of the joint
spectrum is given by the equation
\[
S(x_1,...,x_{n-1})=\det(x_1A_1+...+x_{n-1}A_{n-1}-I)=0,
\]
and we obtain that in this case the spectral representation
is a ``true'' determinantal representation. In particular,
when all the operators are of finite rank, the joint
spectrum is a classical determinantal hypersurface in
$\C{\mathbb P}^{n-1}$. Of course, for infinite rank operators
the analyticity holds only on an open subset of
$\C{\mathbb P}^{n-1}$ and that moves the problem of
describing properties of the joint spectrum from the area
of  projective geometry to analytic geometry.

We are unaware of any results addressing the question
of when a spectral representation exists.
We will see in the next section that if the operators
$A_1,\dots,A_{n-1}$ are compact, then the proper part of the
spectrum
$
\sigma_p(A_1,...,A_{n-1})
%=
%\sigma(A_1,...,A_{n-1},I) \cap \{z_n \neq 0\}
$
is a globally determined analytic set of codimension one in
every compact subset of the chart $\{ x_n\neq 0\}$.
It would be interesting to know for which analytic sets
that are zeros of entire functions   spectral
representations exist. We will also show that if $(x_0,y_0)$
is a point in the joint spectrum of $A_1$ and $A_2$ such
that $x_0A_1+y_0A_2$ is normal and $1$ is an isolated
spectral point of $x_0A_1+y_0A_2$ of finite multiplicity,
then the joint spectrum is an analytic set in a
neighborhood of $(x_0,y_0)$.

A main goal of our paper is to address the relationship
between the geometry of the spectrum
and the mutual behavior of the operators in the
pencil. There is a recent result \cite{CSZ}
which is a generalization of the Motzkin-Taussky theorem
mentioned above to the infinite dimensional case. This
result states that a tuple $(A_1,\dots,A_n)$ of self-adjoint
compact operators acting on a separable Hilbert space
commute pairwise if and only if their proper
joint spectrum $\sigma_p(A_1,\dots,A_n)$ is
a locally finite union of affine hyperplanes (of course,
local finiteness is not a condition but just a property
coming from compactness).

Clearly, if the operators $A_1,\dots,A_n$ have a common
eigenvector with corresponding eigenvalues
$\lambda_1,\dots,\lambda_n$, then the proper joint spectrum
$\sigma_p(A_1,\dots,A_n)$ contains a hyperplane
$\{\lambda_1x_1+ \dots +\lambda_n x_n=1\}$. More generally,
if these operators have a common invariant subspace  $L$
of dimension $k$ (so the corresponding tuple is decomposable
with one block having dimension $k$), then the proper
spectrum contains an algebraic hypersurface of degree $k$
given by
\[
\det\bigl(x_1 (A_1|_L)+ \dots + x_n (A_n|_L)\bigr)=0.
\]
It is natural to ask whether the converse is true: if the
proper spectrum of self-adjoint operators $A_1,\dots,A_n$
contains an algebraic hypersurface of degree $k$, does
this imply that those operators have a common invariant
$k$-dimensional subspace, so that the tuple is decomposable?
In general the answer is negative as shown by the following
simple example.  Consider
$$
\Gamma =\{(x,y)\in \C^2: \ (x+y-1)(5xy+5 y^2-15y-10x+2)=0\}
$$
and let
\begin{equation}\label{matrices}
A_1=
\left[
\begin{array}{ccc}
1 & 0 & 0 \\
0 & 5 & 0 \\
0 & 0 & 0
\end{array}
\right],
\ A_2 =
\left[
\begin{array}{ccc}
1 & 2 & 1 \\
2 & 7 & 1 \\
1 & 1 & 1/2
\end{array}
\right].
\end{equation}
It is easy to check that $\Gamma$ is the set of zeros of
the determinant of the monic pencil $xA_1+yA_2-I$, and,
therefore, is the proper joint spectrum of $A_1$ and $A_2$.
Note that $\Gamma$ has two irreducible
components: the line $\{ x+y=1\}$ and the quadratic
$\{5xy+5y^2-15y-10x+2=0\}$; but $A_1$ and $A_2$ have neither
a common eigenvector nor a common two-dimensional invariant
subspace (in fact  $\Gamma$ has no decomposable
determinantal representation as it can also be shown using
one of the results in \cite{KV}).

Our first main task is to give necessary and
sufficient \emph{geometric} conditions for the presence of
an algebraic curve in the proper joint spectrum to indicate
decomposability, that is the existence of a common
invariant subspace of dimension equal to the degree
of the curve. Our consideration will be concentrated on a
local analysis near an isolated spectral point of one of
the operators in the tuple. In particular, when one of the
opeartors in the tuple is self-adjoint and is a sum of
compact and real multiple of the identity, the spectrum of
such an operator has at most one accumulation poin and our
consideration is valid near every spectral point of this
operator excluding the accumulation one.

We also address the issue of spectral continuity, that is,
if two curves are close in a neighborhood of a point,
both have self-adjoint spectral representation of which
one is decomposable, how far from being decomposable is
the other? The specific question we are considering is:
given that the proper joint spectrum is  close to a line
 in a neighborhood of one of its points,
does this mean that the operators have a common
``almost eigenvector'' (common ``almost invariant''
subspace)? Results in Sections~\ref{S:spectral-continuity}
and \ref{S:norm-estimates} present conditions
that guarantee that this is true, and give some norm 
estimates.

Let us now describe in more detail the main results of 
the paper.
The first important case of the problem when an algebraic
curve in the proper joint spectrum is associated with a
common invariant subspace is the case of a spectral line.
This case turns out to be crucial for higher order spectral
curves. The following result is proved in
Section~\ref{S:line-in-spectrum} (here, as well as in the rest
of the paper, %for a point $x\in\C^n$
we denote by $\Delta_\rho(x)$ the polydisk of radius $\rho$
centered at $x\in\C^n$).

\begin{theoremA}
Let $A_1,\dots,A_n$ be self-adjoint, $\lambda \neq 0$ be an
isolated point of $\sigma(A_1)$, and there exists $\rho>0$
such that, up to multiplicity,
\small{
\[
\begin{aligned}
\Delta_\rho( 1/\lambda ,0,\dots,0) & \cap
\{
\lambda x_1+a_2x_2+ \dots +a_nx_n=1
\}
\\
=
\Delta_\rho( 1/\lambda ,0,\dots,0) & \cap
\sigma_p(A_1,\dots, A_n).
\end{aligned}
\]
}
\negthickspace
The following are equivalent:
\begin{itemize}
\item[(1)]
The eigensubspace of $A_1$ corresponding to eigenvalue
$\lambda$ is an eigensubspace for each of the operators
$A_2,\dots, A_n$;

\item[(2)]
There exist an $\epsilon \in \R, \ \epsilon \neq 1$, and
$\rho^{\prime}>0$ such that $A_1(\epsilon,\lambda)$ is
invertible and, up to multiplicity,
%the plane segment
\small{
\[
\begin{aligned}
\Delta_{\rho^\prime}(\lambda,0,\dots, 0) & \cap
\{
(1/\lambda)x_1+a_2x_2+ \dots +a_nx_n=1
\}
\\
=
\Delta_{\rho^\prime}(\lambda,0,\dots, ,0) & \cap
\sigma_p\bigl(
 A_1(\epsilon,\lambda)^{-1},
 A_2(\epsilon,a_2),\dots,A_n(\epsilon,a_n)
\bigr),
\end{aligned}
\]
}
%\negthickspace
\negthickspace
where
$A(\epsilon, b)=(1+\epsilon )A -b\epsilon I.$ %\\
\end{itemize}
\end{theoremA}
%\vspace{.2cm}

The most important case here
is the one of two operators. 
Theorem~B below is obtained from this case
by passing to tensor powers of operators and considering
their action on the exterior power of the corresponding
Hilbert space. Here for an operator $A$ acting on a Hilbert
space $H$ we write $\wedge ^n A$ to indicate that we
consider the action of $\otimes ^n A$ on $\wedge ^n H$.
We say that a self-adjoint operator $A$ on a separable
Hilbert space $H$ belongs to the class $\mathcal E(H)$ if
$A=K + aI$ for some compact self-adjoint operator $K$ and
some $a\in \R$. It is shown in
Section~\ref{S:exterior-powers}
that for operators $A$ and $B$ in $\mathcal E(H)$ one can
always use an appropriate change of coordinates to
reduce the search for common invariant subspaces
to the ``general position'' setting considered in
our next main result:

\begin{theoremB}
Let $A=K_1+ aI$ and $B=K_2+bI$ be self-adjoint operators
in the class $\mathcal E(H)$, with $A$ invertible.
Let $\Gamma$ be an algebraic curve of degree $k$ which is a
union of components 
of the proper joint spectrum $\sigma_p(A,B)$,
and which does not have the line $\{ax + by = 1 \}$ as a
reduced component.
Suppose that the $x$-axis (resp. the $y$-axis) intersects
$\Gamma$ in the $k$ points (counted with multiplicity)
$1/\lambda_1,\dots, 1/\lambda_k$
(resp. $1/\mu_1,\dots,1/\mu_k$)
such that each point $(1/\lambda_i,0)$ belongs only to
components of $\sigma_p(A,B)$ contained in $\Gamma$,
Set $\lambda=\lambda_1\dots\lambda_k$ and
$\mu=\mu_1\dots\mu_k$, and suppose that $\lambda$ is an
isolated eigenvalue of multiplicity $1$ in the spectrum
of $\wedge^k A$.
The following are equivalent:
\begin{itemize}
\item[(1)]
The eigenspace for $A$ corresponding to
$\lambda_1,\dots,\lambda_k$ is invariant for $B$.

\item[(2)]
There exists $\rho>0$ such that
the line segments
\[
\begin{aligned}
\{\lambda x + \mu y =1\} & \cap\Delta_{\rho}(1/\lambda, 0)
\text{ and } \\
\{(1/\lambda)x+ \mu y=1\}& \cap\Delta_{\rho}(\lambda, 0)
\end{aligned}
\]
are contained in 
$
\sigma_p\bigl(
\wedge^k A, \wedge^k B
\bigr)
$
and 
$
\sigma_p\bigl(
\wedge^k A^{-1}, \wedge^k B
\bigr),
$
respectively.

\item[(3)]
The lines
\begin{equation*}
\{\lambda x + \mu y =1\}
\quad\text{ and }\quad 
\{(1/\lambda)x+ \mu y=1\} 
\end{equation*}
are contained in 
$
\sigma_p\bigl(
\wedge^k A, \wedge^k B
\bigr)
$
and 
$
\sigma_p\bigl(
\wedge^k A^{-1}, \wedge^k B
\bigr),
$
respectively.
\end{itemize}
\end{theoremB}

An extension of this result to a tuple of arbitrary length
holds here as well, and is derived from Theorem B exactly
the same way as the 
result of Theorem A is derived from the
coresponding result for two operators. For this reason its
precise statement is omitted.
Application to the classical setting when
$A_1,\dots,A_n$ are self-adjoint operators on $\C^N$
yields the following theorem (here, as explained in
Section~\ref{S:exterior-powers}, for an invertible self-adjoint
operator $A$ we consider $\wedge^{N-k}A$ in a natural
way as an operator on $\wedge^k\C^N$, and
again, we can always reduce
to the ``general position'' situation considered below).

\begin{theoremC}
Let $C$ be a reducible real algebraic hypersurface
of degree $N$ in $\C^n$, and let $\Gamma$ be a
degree $k$ hypersurface that is a union of components of
$C$, such that for each $i$ the $x_i$-axis intersects $\Gamma$
in the $k$ points  $1/\alpha_{i1},\dots, 1/\alpha_{ik}$,
counted with multiplicities.
Let $a_i= \alpha_{i1}\dots\alpha_{ik}$ and
suppose also that each point
$1/\alpha_{1j}$ belongs only to components of $C$ contained
in $\Gamma$.
Let a tuple $(A_1,\dots,A_n)$ consisting of
self-adjoint operators on $\C^N$, with $A_1$ invertible,
be a determinantal representation of $C$,
and suppose that $a_1$ is an eigenvalue of multiplicity
$1$ for $\wedge^k A_1$.

This
representation induces a determinantal
representation of $\Gamma$
if and only if the hypersurface
\[
\{
\det(x_1 \wedge ^k A_1+\dots + x_n\wedge ^k A_n - I)=0
\}
\]
contains the hyperplane
$\{ a_1 x_1+ \dots + a_nx_n = 1\}$,
and the hypersurface
\[
\{
\det( x_1\wedge^{N-k} A_1 + x_2\wedge^k A_2 + \dots +
      x_n\wedge^k A_n - I) = 0
\}
\]
contains the hyperplane
$
\{
(\det A_1/a_1)x_1 + a_2x_2 + \dots + a_nx_n = 1
\}.
$
\end{theoremC}

Now we turn to spectral continuity.
For a positive $\epsilon$ we say that a vector $\xi$ is an
$\epsilon$-eigenvector of an operator $A$
(almost eigenvector) if there exists $\lambda$ such that
\[
\left\Vert A\xi -\lambda\xi \right\Vert <
\epsilon \left\Vert \xi \right\Vert.
\]
Our first result regarding spectral continuity,
Theorem~\ref{close to line}, states that,
under some natural assumptions, if the joint spectrum of a
pair $(A_1,A_2)$ of self-adjoint operators, with $A_1$
invertible, is
$\epsilon$-close in the Hausdorff
metric to a line $\{ \alpha x+ \beta y=1\}$ in a neighborhood
of an isolated spectral point of $A_1$,
and the same is true for the joint spectrum of the pair
$(A_1^{-1}, A_2)$, then they have a
common almost eigenvector of order $\sqrt{\epsilon}$.
If $|\beta|=\Vert A_2\Vert$, the condition on the
joint spectrum of $A_1^{-1}$ and $A_2$ can be omitted.
As a corollary to this result we obtain the following
estimate for the commutant of two self-adjoint matrices.

\begin{theoremD}
Let $A_1$ and $A_2$ be two self-adjoint $N\times N$ matrices
with eigenvalues $\alpha_1,\dots,\alpha_N$ and
$\beta_1,\dots,\beta_N$
respectively, satisfying
$|\alpha_1|>\dots>|\alpha_N|>0$ and 
$|\beta_1|>\dots>|\beta_N|>0$.
Suppose that $\ell_1,\dots,\ell_N$ is a family of lines,
\[
\ell_j=
\{ \alpha_{n(j)}x+\beta_j y=1\}, \ 1\leq j, n(j)\leq N, 
\]
such that:  
\begin{itemize}
\item[(1)]
each of the points $(\frac{1}{\alpha_k},0), \ 1\leq k\leq N$
belongs to one of these lines;

\item[(2)]
there exist $0<\rho<1$ and $0< \epsilon \ll \rho$ such that
conditions (1) and (2) of Theorem~\ref{almost} are true for
$\sigma_p(A_1,A_2)$ and each %the line
$\ell_j$.
\end{itemize}
Then if $\epsilon$ is small enough, the norm of the
commutant of $A_1$ and $A_2$ is at most of order
$\epsilon ^{1/2^N}$.
\end{theoremD}

The structure of this paper is as follows.
Section~\ref{S:determining-functions} is devoted to
determining functions.
In Section~\ref{S:conditions-algebraic-curve} we find
conditions for an algebraic curve to be a part of
the proper joint spectrum. These conditions are expressed
in terms of holomorphy of a sequence of certain
operator-valued functions.
In Section~\ref{S:line-in-spectrum}
we prove Theorem~A.
Theorems~B and C are proved in Section~\ref{S:exterior-powers}.
%Here we also have
In Section~\ref{S:example} we give an
example of a circle in the spectrum, where the existence of
a common reducing subspace is derived not from exterior
products, but directly from the vanishing residues conditions
found in Section~\ref{S:conditions-algebraic-curve}.
Section~\ref{S:spectral-continuity} is devoted to spectral
continuity. Theorem~D is proved in Section~\ref{S:norm-estimates}.
Finally, Section~\ref{S:concluding-remarks}
contains several concluding remarks and
open questions.

%{\bf Acknowledgement}
We would like to thank R.Yang for useful comments.

\section{Determining functions}
\label{S:determining-functions}

Let $A_1$ and $A_2$ be bounded  operators on a Hilbert space
$H$. Recall that the \emph{proper part} of the projective
joint spectrum
$\sigma(A_1,A_2,I)$, or just the \emph{proper projective
spectrum} $\sigma_p(A_1,A_2)$  is the following set:
\[
\sigma_p(A_1,A_2)=
\{ (x,y)\in \C^2: \ (x,y,-1)\in \sigma(A_1,A_2,I)\}.
\]
%}

It was shown in \cite{SYZ} that if $A_1$ and $A_2$ are
compact, then $\sigma_p(A_1,A_2)$ is an analytic set of
codimension one in $\C^2$. The following explicit
construction of the analytic function locally determining
this set was given there. We present it for tuples of
self-adjoint operators: the setting we consider in this paper.

Suppose that $A_1,\dots ,A_m$ is a tuple of 
compact self-adjoint
operators on a Hilbert space $H$.  Choose a small
$\epsilon>0$ and finite rank self-adjoint operators
$K_1,\dots,K_m$ such that
$\left\Vert A_j- K_j \right\Vert <\epsilon$.
If $(w_1,\dots, w_m)\in \C^m$ satisfy
$\sum_{j=1}^m |w_j|<\frac{1}{\epsilon}$, then the operator
$I-\sum_{j=1}^mw_j(A_j-K_j)$ is invertible and we have
\small{
\begin{eqnarray}
\sum_{j=1}^m w_jA_j -I=
\sum_{j=1}^m w_jK_j-I+ \sum_{j=1}^mw_j(A_j-K_j) \nonumber \\
=
\left(
I-\sum_{j=1}^mw_j(A_j-K_j)
\right)
\left(
\sum_{l=1}^m w_l \left( I-\sum_{j=1}^mw_j(A_j-K_j) \right)^{-1}
               K_l - I
\right). \nonumber
\end{eqnarray}
}

Thus, $(w_1,\dots,w_m)\in \sigma_p(A_1,\dots,A_m)$ 
if and only if the operator
$$
\sum_{l=1}^m
w_l\left( I-\sum_{j=1}^mw_j(A_j-K_j) \right)^{-1} K_l-I
$$
is not invertible. Since
$
\sum_{l=1}^m
w_l\left( I-\sum_{j=1}^mw_j(A_j-K_j) \right)^{-1}K_l
$
is of finite rank, there is a finite dimensional subspace
$L$ of $H$ such that this operator vanishes on the
complement to this subspace and is represented by an
$n\times n $ matrix on this subspace.
Therefore, this operator is not invertible if and only if
\begin{equation}\label{determine}
\det\left(
\sum_{l=1}^m
w_l\left( I-\sum_{j=1}^mw_j(A_j-K_j) \right)^{-1}K_l -I \right)
=0.
\end{equation}
The left-hand side of \eqref{determine} is clearly an
analytic function of $w_1,\dots,w_m$ in the domain
$\{ \sum_{j=1}^m |w_j|<\frac{1}{\epsilon} \}$ and
\eqref{determine}  determines $\sigma_p(A_1,\dots,A_m)$ in
this domain. We call this function a
\emph{determining function} of the proper projective spectrum.
It follows that a different choice of the finite rank
approximations leads to a determining function with the
same divisor of zeros in
$
\{(w_1,\dots,w_m)\in \C^m: \
\sum_{j=1}^m |w_j|<\frac{1}{\epsilon}
\}
$.

If $A_1$ and $A_2$ are not compact, the joint spectrum is
not necessarily an analytic set. For example, if $A_1=I$ is
the identity operator, the joint spectrum is a cone with
vertex at $(1,0)$ that consists of lines
$\{ x+\lambda y=1\}, \ \lambda \in \sigma(A_2)$. Thus,
if the cardinality of $\sigma(A_2)$ is infinite, the
joint spectrum is not analytic at $(1,0)$. Nevertheless,
essentially the same argument we used above to show the
analyticity of the joint spectrum in the compact case,
establishes the following local result.

\begin{theorem}
Let $A_1$ and $A_2$ be bounded operators on $H$, with $A_1$
normal, and $\lambda\neq 0$ be an isolated point of
$\sigma(A_1)$ of finite multiplicity. Then
$\sigma_p(A_1,A_2)$ is an analytic set in a neighborhood of
$(\frac{1}{\lambda},0)$.
\end{theorem}

\begin{proof}
The spectral decomposition of $A_1$ is in the form
$$
A_1=\lambda P_1+\int_{\sigma(A_1)\setminus \{ \lambda\}} zdE(z),
$$
where the operator $P_1$ is the finite
rank orthogonal projection of
$H$ onto the eigenspace of $A_1$ with eigenvalue $\lambda$
and $dE$ is the spectral measure on the rest of
$\sigma(A_1)$. Since $\lambda$ is an isolated spectral
point of $A_1$, if $(x,y)$ is close to
$(\frac{1}{\lambda},0)$, the operator
$$
\tilde{A}(x,y)=
x\int_{\sigma(A_1)\setminus \{ \lambda\}}zdE(z) + yA_2 - I
$$
is inverible. Therefore, such a point $(x,y)$ belongs to
the joint spectrum if and only if the operator
$$
B(x,y)=xP_1\tilde{A}(x,y)^{-1}-I
$$
is not invertible. Since $xP_1\tilde{A}(x,y)$ has finite rank
$n$ which is equal to the rank of $P_1$, the pairs $(x,y)$
for which $B(x,y)$ is not invertible are zeros of a
determinant of an $n\times n$ matrix, whose coefficients
are analytic functions of $(x,y)$, and the result follows.
\end{proof}

If the multiplicity of an isolated spectral point
$\lambda \in \sigma(A_1)$ is equal to one, the local
analyticity of the joint spectrum holds even without
$A_1$ being  normal.

Let $A$ be a bounded operator acting on $H$, and let
$\lambda $ be an isolated spectral point of $A$. Recall
that $\lambda$ is said to have multiplicity $k$ if for a
contour $\gamma$ in the resolvent set that contains
$\lambda$ as the only spectral point of $A$
\begin{equation}\label{contour_integral}
P_\lambda=\frac{1}{2\pi i}\int_\gamma (wI-A)^{-1}dw,
\end{equation}
is a rank $k$ projection (not necessarily orthogonal).

\begin{theorem}\label{analytic_near_isolated_point}
Let $A_1$ and $A_2$ be  operators on $H$ and
$\lambda\neq 0$ be an isolated spectral point of $A_1$
of multiplicity one. Then there exists $\rho >0$ such
that in $\Delta_\rho(\frac{1}{\lambda},0)$ the proper joint
spectrum $\sigma_p(A_1,A_2)$ is a nonsingular analytic set.
\end{theorem}

\begin{proof}
If $\rho$ is small enough and
$(x,y)\in \Delta_\rho(\frac{1}{\lambda},0)$, the operator
$A(x,y)=xA_1+yA_2$ has an isolated spectral point
$\lambda(x,y)$ close to $1$
that also has multiplicity one,
so the projection
$$
P(x,y)=\frac{1}{2\pi i}\int_\gamma (wI-A(x,y))^{-1}dw
$$
has rank one, cf. \cite[p. 13]{GK}, and the range of $P(x,y)$
consists of eigenvectors of $A(x,y)$ with eigenvalue
$\lambda(x,y)$. The joint spectrum of $A_1$ and $A_2$ consists
of those pairs $(x,y)$ for which $\lambda(x,y)=1$. Let $e$ be
the unit eigenvector of $A_1$ with eigenvalue $\lambda$.
Then for $(x,y)\in \Delta_\rho(\frac{1}{\lambda},0)$ we have
that $P(x,y)e$ is close to $e$, and, therefore,
$P(x,y)e\neq 0$. Now, $\lambda(x,y)=1$ if and only if
$A(x,y)P(x,y)e=P(x,y)e$ and that happens if and only if
\begin{equation}\label{belonging_to_spectrum}
\langle (A(x,y)P(x,y)-P(x,y))e,e\rangle =0.
\end{equation}
Equation \eqref{belonging_to_spectrum} determines
$\sigma_p(A_1,A_2)$  near the point $(\frac{1}{\lambda},0)$
and it is easily seen that the left-hand side is analytic
in $x$ and $y$.   Now we write down explicitly the Taylor
decomposition of this function in terms of
$\Delta x=x-\frac{1}{\lambda}$ and $y$. We have
\small{
\begin{eqnarray}
A(x,y)P(x,y)-P(x,y)=
\frac{1}{2\pi i}\int_\gamma (w-1)(wI-A(x,y))^{-1}dw
\nonumber \\
=
\frac{1}{2\pi i}\int_\gamma (w-1)(wI-\frac{1}{\lambda}A_1)^{-1}
\left(
I-(\Delta xA_1+yA_2)(wI-\frac{1}{\lambda}A_1)^{-1}
\right)^{-1}dw
\nonumber \\
=
\sum_{j=0}^\infty
\frac{1}{2\pi i}\int_\gamma (w-1)(wI-\frac{1}{\lambda}A_1)^{-1} \left[
(\Delta xA_1+yA_2)(wI-\frac{1}{\lambda}A_1)^{-1}
\right]^j dw
\nonumber \\
=
\sum_{k,m=0}^\infty (\Delta x )^ky^m
\left(
\frac{1}{2 \pi i}
\int_\gamma (w-1)(wI-\frac{1}{\lambda}A_1)^{-1}
                            {\mathcal D}_{k,m}(w)dw
\right),
\nonumber
\end{eqnarray}
}
\negthickspace
\negthickspace
where
$$
{\mathcal D}_{k,m}(w)=
\sum_\alpha \prod_{l=1}^{k+m} {\mathcal S}_{\alpha_l}(w),
$$
with summation taken over all sequences
$\alpha=(\alpha_1, \dots ,\alpha_{k+m})$ of zeros and
ones of length $(k+m)$ having $k$ zeros and $m$ ones, and
$$
{\mathcal S}_0=
A_1(wI- \frac{1}{\lambda}A_1)^{-1}, \ \quad
{\mathcal S}_1=
A_2(wI- \frac{1}{\lambda}A_1)^{-1}.
$$
Thus we have
\begin{equation*}
\sigma_p(A_1,A_2)\cap \Delta_\rho (\frac{1}{\lambda},0) =
\left\{
(x,y)\in \Delta_\rho(\frac{1}{\lambda},0):
{\mathcal F}(x,y)= 0
\right\},
\end{equation*}
where
\small{
\begin{equation*}
{\mathcal F}(x,y) =
\sum_{k,m=0}^\infty (x-\frac{1}{\lambda})^ky^m
\left(
\frac{1}{2 \pi i}
\int_\gamma (w-1)
\langle
(wI-\frac{1}{\lambda}A_1)^{-1}{\mathcal D}_{k,m}(w)e,e
\rangle
dw
\right)
\end{equation*}
}
\negthickspace
\negthickspace
Obviously, ${\mathcal F}(x,y)$ is a nontrivial analytic
function, so the joint spectrum is an analytic set in
$\Delta_\rho (\frac{1}{\lambda},0)$. Further, it follows
directly from the Taylor decomposition above that
\begin{eqnarray}
\frac{\partial {\mathcal F}}{\partial x}
\bigg|_{x=\frac{1}{\lambda}, y=0} =
\frac{1}{2\pi i}\int_\gamma (w-1)
\langle
(wI-\frac{1}{\lambda}A_1)^{-1}A_1
(wI-\frac{1}{\lambda}A_1)^{-1}e, e
\rangle
dw
\nonumber \\
=
\frac{\lambda}{2 \pi i}\int_\gamma \frac{dw}{w-1} =
\lambda\neq 0,
\nonumber
\end{eqnarray}
and, therefore, the zero set of ${\mathcal F}$ is
nonsingular near $(\frac{1}{\lambda},0)$. %We are done.
\end{proof}

\section{
Necessary conditions for an algebraic
curve in the joint spectrum
}
\label{S:conditions-algebraic-curve}

Let $A_1$ and $A_2$ be self-adjoint operators and 
$\lambda \neq 0$ be an
isolated point of $\sigma(A_1)$ such that
\begin{itemize}
\item[(a)]
$\sigma_p(A_1,A_2)$ in a neighborhood
$\Delta_\rho(\frac{1}{\lambda},0)$ of
$(\frac{1}{\lambda},0)$ is an algebraic curve
given by a polynomial equation ${\mathcal R}(x,y)=0$
of degree $k$, where ${\mathcal R}$ is a polynomial
with real coefficients, that is
$$
\sigma_p(A_1,A_2)\cap
\Delta_\rho \left( \frac{1}{\lambda},0\right) =
\{ (x,y)\in \Delta_\rho\left(\frac{1}{\lambda},0 \right): \
{\mathcal R}(x,y)=0\};
$$

\item[(b)]
$(\frac{1}{\lambda},0)$
belongs to only one reduced component of this curve and
is a nonsingular point on this reduced component;

\item[(c)]
the axis $\{ y=0\}$ is not tangent to
this reduced component of
the curve
at $(\frac{1}{\lambda},0)$.
\end{itemize}
Here for a curve, or more generally, for a hypersurface
defined by
a polynomial $G=G_1^{r_1}\dots G_m^{r_m}$
(with each polynomial $G_i$ irreducible and $G_i$ not associate
with $G_j$ for $i\ne j$),
the \emph{components} of that hypersurface are defined by the
polynomials $G_i^{r_i}$ (thus they are irreducible
but not necessarily reduced), and the
\emph{reduced components} are defined by the
polynomials $G_i$.
Write
\begin{equation}\label{polynom}
{\mathcal R}(x,y) =
\sum_{j=0}^k   R_j(x,y), \ \text{where} \
R_j=\sum_{m=0}^j r^j_m x^my^{j-m}, \ \text{and} \ R_0=-1.
\end{equation}
Passing to a smaller neighborhood if necessary,
we may assume that
\begin{itemize}
\item[(i)]
the reduced component
containing $(1/\lambda, 0)$
of the curve $\{ {\mathcal R}(x,y)=0\}$
has no singular points
in $\Delta_\rho(\frac{1}{\lambda},0)$;

\item[(ii)]
there is $0<\rho'<\rho$ such that
for $(x,y)\in \Delta_{\rho'}(\frac{1}{\lambda},0)$ the complex
line $\{(\tau x,\tau y): \ \tau\in \C\}$ has %exactly
(up to multiplicity) exactly
one point of intersection with $\{ {\mathcal R}(x,y)=0\}$
that lies in $\Delta_\rho(\frac{1}{\lambda},0)$.
\end{itemize}

Let $(x,y)\in \Delta_{\rho'}(\frac{1}{\lambda},0)$ and
$(\tau x, \tau y)\in  \{ {\mathcal R}(x,y)=0\}$. Then
${\mathcal R}(\tau x, \tau y)=0$, and
the equation in $\tau$
\[
\tau^k R_k(x,y)+\tau^{k-1}R_{k-1}(x,y)+\dots +\tau R_1(x,y) -1=0
\]
has exactly one root, $\tau(x,y)$, in a neighborhood of $1$.
The corresponding eigenvalue $\mu(x,y)= 1/\tau(x,y)$
of the operator $xA_1+yA_2$
satisfies the equation
\begin{equation}\label{eigenvalue}
\mu^k - \mu^{k-1}R_1(x,y)- \ldots -R_k(x,y)=0.
\end{equation}
Of course, $\mu(x,y)$ is the only eigenvalue  of
$xA_1+yA_2$ which lies at distance of order $\rho$
from $1$ and is an isolated point of the spectrum
$\sigma(xA_1+yA_2)$. It is also clear that if $\lambda$
is a multiple spectral point of $A_1$, then $\mu(x,y)$
has the same multiplicity.

If both $x$ and $y$ are real,  $xA_1+yA_2$ is self-adjoint.
Let $\zeta(x,y)$ be an eigenvector of $xA_1+yA_2$ with
eigenvalue  $\mu (x,y)$. Then equation \eqref{eigenvalue}
implies
\begin{equation}\label{vanish}
[(xA_1+yA_2)^k-R_1(x,y)(xA_1+yA_2)^{k-1}- \ldots -R_k(x,y)]
\zeta (x,y)=0.
\end{equation}
Let $L(x,y)\subset H$ be the eigensubspace of
$xA_1+yA_2$ corresponding to $\mu (x,y)$, and let
$P(x,y):H\to L(x,y)$ be the orthogonal projection.
For $0<\delta<\tau$ write
$\gamma =\{ z\in \C : \ |z-1|<\delta \}$. We have
\begin{equation}\label{projection}
P(x,y)=
\frac{1}{2\pi i}\int_{\gamma} (wI - (xA_1+yA_2))^{-1}dw.
\end{equation}
It is readily seen that for $m=0,1,2,\dots$
\begin{equation}\label{power}
(xA_1+yA_2)^m P(x,y) =
\frac{1}{2\pi i}\int_{\gamma} w^m (w I-(xA_1+yA_2))^{-1}dw.
\end{equation}
Equations \eqref{vanish} and \eqref{power} imply that
for every $(x,y)$ sufficiently close to
$(\frac{1}{\lambda},0)$
the following identity holds:
\[
\frac{1}{2\pi i}\int_{\gamma}
\left[ w^k -\sum_{j=1}^k R_j(x,y)w^{k-j}\right]
(w I -(xA_1+ yA_2))^{-1}dw=0.
\]
Write $\Delta x= x-\frac{1}{\lambda}$. If $\Delta x$
and $y$ are sufficiently small,
the last relation implies
\begin{multline}\label{decomp2}
\frac{1}{2\pi i}\int_{\gamma}
\left[ w^k-\sum_{j=1}^k R_j(x,y)w^{k-j} \right]
(w I - \frac{1}{\lambda}A_1)^{-1}
\\
\times \sum_{m=0}^\infty
\left[
(\Delta xA_1+ yA_2)(w I - \frac{1}{\lambda}A_1)^{-1}
\right]^m
dw =0.
\end{multline}
If $\Delta x=0$, the last relation turns into the following:
\begin{multline*}
\frac{1}{2\pi i}\int_{\gamma}
\left[
w^k-\sum_{j=1}^k w^{k-j}\sum_{n=0}^j r_{j-n}^jy^nx_1^{j-n}
\right]
(w I - x_1A_1)^{-1}
\\
\times  \sum_{m=0}^\infty y^m
\left[ A_2(w I -  \frac{1}{\lambda}A_1)^{-1}  \right]^m dw =0.
\end{multline*}
Rearranging terms in the last equation, we obtain
\small{
\begin{equation}\label{decomp4}
\begin{split}
\sum_{m=0}^{k-1}
\frac{y^m}{2\pi i}\int_\gamma
\Biggl\{
\biggl(
&
w I - \frac{1}{\lambda}A_1
\biggr)^{-1}
\\
\times
\Biggl(
&
\biggl(
w^k \sum_{j=1}^k w^{k-j}\frac{r_j^j}{\lambda^j}
\biggr)
\left[
A_2(w I- \frac{1}{\lambda} A_1)^{-1}
\right]^m
\\
-
&
\sum_{n=1}^m
\biggl(
\sum_{j=n}^k w ^{k-j}\frac{r_{j-n}^j}{\lambda^{j-n}}
\biggr)
\left[
A_2( w I - \frac{1}{\lambda}A_1)^{-1}
\right]^{m-n}
\Biggr)
\Biggr\}
dw
%\label{decomp4}
\\
+ \sum_{m=k}^\infty \frac{y^m}{2\pi i}\int_\gamma
&
\Biggl\{
\biggl(
w I - \frac{1}{\lambda}A_1
\biggr)^{-1}
\\
\times
&
\Biggl(
\biggl(
w^k - \sum_{j=1}^k w^{k-j}\frac{r_j^j }{\lambda^j}
\biggr)
\biggl[
A_2(w I - \frac{1}{\lambda}A_1)^{-1}
\biggr]^k
\\
&
- \sum_{n=1}^k
\biggl(
\sum_{j=n}^k w^{k-j}\frac{r^j_{j-n}}{\lambda^{j-n}}
\biggr)
\biggl[
A_2(w I - x_1 A_1)^{-1}
\biggr]^{k-n}
\Biggr)
\\
&
\qquad
\times
\biggr[
A_2(w I - \frac{1}{\lambda} A_1)^{-1}
\biggr]^{m-k}
\Biggr\}
dw =0.
%\nonumber
\end{split}
\end{equation}
}
\negthickspace
Since \eqref{decomp4} holds for every $y$ in a neighborhood
of the origin, it implies
%\begin{eqnarray}
\small{
\begin{multline}\label{condition1}
\frac{1}{2\pi i}\int_\gamma
\left\{
(w I -x_1A_1)^{-1}
\left(
\left( w^k - \sum_{j=1}^k w^{k-j}\frac{r_j^j}{\lambda^j }\right)
[A_2(w I -\frac{1}{\lambda}A_1)^{-1}]^m
\right.
\right.
%\nonumber
\\
-
\left. \left.
\sum_{n=1}^m
\left( \sum_{j=n}^k w^{k-j}\frac{r^j_{j-n}}{\lambda^{j-n}}\right)
[A_2(w I - \frac{1}{\lambda}A_1)^{-1}]^{m-n}
\right)
\right\}
dw =0,
\end{multline}
}
\negthickspace
for $1\leq m\leq k-1$, and
%\nonumber \\
\small{
\begin{multline}\label{condition2}
\frac{1}{2\pi i}\int_\gamma
\Biggl\{
\biggl(
w I -\frac{1}{\lambda}A_1
\biggr)^{-1}
\\
\times
\Biggl(
\biggl(
w^k - \sum_{j=1}^k w^{k-j}\frac{r_j^j}{\lambda^j}
\biggr)
\left[
A_2(w I- \frac{1}{\lambda}A_1)^{-1}
\right]^k
\quad
%\qquad
\\
\shoveright{
- \sum_{n=1}^k
\biggl(
\sum_{j=n}^k w^{k-j}\frac{r_{j-n}^j}{\lambda^{j-n}}
\biggr)
\left[
A_2(w I -\frac{1}{\lambda}A_1)^{-1}
\right]^{k-n}
\Biggr)
\qquad\qquad
}
\\
\times
\left[
A_2(w I - \frac{1}{\lambda}A_1)^{-1}
\right]^{m-k}
\Biggr\}
dw = 0
\end{multline}
}
\negthickspace
for $m\geq k$.
The integrands in \eqref{condition1} and
\eqref{condition2}
are operator-valued holomorphic functions
in the punctured disk
$\{ w \in \C: 0<|w -1|<\delta \}$ with poles at one.
We denote these integrands
by $\Psi_m(w), \ m\geq 1$. Thus, \eqref{condition1} and
\eqref{condition2} imply the following result:

\begin{theorem}\label{Psi_1}
Suppose that $A_1$ and $A_2$ are self-adjoint operators
acting on a separable Hilbert space $H$, with
$\lambda \in \sigma(A_1)$ an isolated point, and an
algebraic curve determined by a polynomial equation
\eqref{polynom} lies in
$\sigma_p(A_1,A_2)$ and satisfies conditions (a) - (c)
above. Then the integrands $\Psi_m(w)$ of
\eqref{condition1} and \eqref{condition2}
satisfy the equation
\[
\frac{1}{2\pi i}\int_\gamma \Psi_m(w)dw =0, \ m\geq 1,
\]
which is equivalent to
\begin{equation}\label{vanishing_residue}
{\mathcal R}es_{w=1}(\Psi_m)=0, \ m\geq 1.
\end{equation}
\end{theorem}

\begin{remark}
%{\bf Remark}.
If the operators $A_1$ and $A_2$ are not self-adjoint,
in general, the result of Theorem~\ref{Psi_1} does not
hold since the  range of the operator \eqref{projection}
does not necessarily consist of eigenvectors of $xA_1+yA_2$.
However, if  $\lambda$ is an isolated spectral point of
$A_1$ of multiplicity one, then for $(x,y)$ sufficiently
close to $(\frac{1}{\lambda},0)$, the operator $P(x,y)$
is a rank one projection (not necessarily orthogonal),
and its range is an eigensubspace of $xA_1+yA_2$, so the
result of Theorem~\ref{Psi_1} is valid in this case too.
\end{remark}

%\vspace{.2cm}
It follows directly from \eqref{condition1} and
\eqref{condition2} that $\Psi_m$ has a pole of order
at most $m+1$ at $w =1$. We will now obtain the
expression of the residue of $\Psi_m$ at $1$.
Let
\begin{equation}\label{spectral_decomposition}
A_1=\lambda P_1+\int_{\sigma(A_1)\setminus \{ \lambda\}} zdE(z)
\end{equation}
be the spectral decomposition of $A_1$ with $P_1$ being
the orthogonal projection on the eigenspace of $A_1$
corresponding to eigenvalues $\lambda$.
If $\delta$ is small enough,  we have
\begin{eqnarray}
&&
(w I - \frac{1}{\lambda}A_1)^{-1}=
\frac{1}{w -1}P_1+
\int_{\sigma(A_1)\setminus \{ \lambda\}}
\frac{dE(z)}{w-\frac{z}{\lambda}} \nonumber \\
&&
= \frac{1}{w -1}P_1-\int _{\sigma(A_1)\setminus \{ \lambda\}}
\left(
\sum_{m=0}^\infty
\left( \frac{\lambda}{z-\lambda}\right )^{m+1}(w-1)^m \right)
dE(z)
\label{inverse1} \\
&&
= \frac{1}{w -1}P_1-\sum_{m=0}^\infty (w-1)^m
\left(
\int _{\sigma(A_1)\setminus \{ \lambda\}}
\left( \frac{\lambda}{z-\lambda}\right)^{m+1}dE(z)
\right). \nonumber
\end{eqnarray}
Write
\begin{equation}\label{operatorT}
T(A_1) = T =
\int_{\sigma(A_1)\setminus\{\lambda\}}\frac{\lambda}{z-\lambda}dE(z),
\end{equation}
then
\[
\int_{\sigma(A_1)\setminus \{ \lambda\}}
\left( \frac{\lambda}{z-\lambda} \right)^{m+1} dE(z) = T^{m+1},
\]
so \eqref{inverse1} can be written as
\begin{equation}\label{inverse2}
(w I-\frac{1}{\lambda}A_1)^{-1} =
\frac{1}{w -1}P_1 - \sum_{m=0}^\infty (w -1)^m T^{m+1},
\end{equation}
and
\begin{equation}\label{block}
A_2(w I- \frac{1}{\lambda}A_1)^{-1} =
\frac{1}{w -1} A_2P_1-\sum_{m=0}^\infty (w -1)^m A_2T^{m+1}.
\end{equation}

The following result follows from
Theorem~\ref{Psi_1} and equations
\eqref{inverse2} and \eqref{block}.

\begin{theorem}\label{Psi_2}
Under the conditions of Theorem~\ref{Psi_1} the associated
integrands $\Psi_m(\lambda)$ determined by \eqref{condition1}
and \eqref{condition2} are holomorphic in
$\{ w \in \C : |w -1|<\delta \}$.
\end{theorem}

\begin{proof}
It follows from \eqref{condition1} and
\eqref{condition2} that
\begin{eqnarray}
&&
\Psi_m(w) =
\Psi_{m-1}(w)[A_2(w I -\frac{1}{\lambda}A_1)^{-1}]
\nonumber \\
&&
-
\left(
\sum_{j=m}^k w^{k-j}\frac{r_{j-m}^j}{\lambda^{j-m}}
\right)
(w I -\frac{1}{\lambda}A_1)^{-1}, \  2\leq m\leq k,
\label{recurrence} \\
&&
\Psi_m(w) =
\Psi_{m-1}(w)[A_2(w I -\frac{1}{\lambda}A_1)^{-1}],\ m\geq k+1.
\nonumber
\end{eqnarray}
Relations \eqref{inverse2}, \eqref{block},  and
\eqref{recurrence} imply that if $\Psi_{m-1}$ is holomorphic,
then $\Psi_m$ has pole of order at most one at $\lambda =1$,
and, therefore, by \eqref{vanishing_residue} $\Psi_m$
is holomorphic. Thus, it suffices to show that
$\Psi_1(w)$ is holomorphic at $w=1$. We have
\begin{eqnarray}
&&
\Psi_1(w) =
(w I -\frac{1}{\lambda}A_1)^{-1}
\left( w^k-\sum_{j=1}^k w^{k-j}\frac{r_j^j}{\lambda^j} \right)
[A_2(w I -\frac{1}{\lambda}A_1)^{-1} ]
\nonumber \\
&&
-
\left(\sum_{j=1}^k w^{k-j}\frac{r^j_{j-1}}{\lambda^{j-1}}\right)
(w I - \frac{1}{\lambda}A_1)^{-1}
=
\tilde{\Psi}_1(w)+\tilde{\tilde{\Psi}}_1(w). \label{psi1}
\end{eqnarray}
Observe that the polynomial
$
{\mathcal P}(w) =
w ^k-\sum_{j=1}^kw^{k-j}\frac{r_j^j}{\lambda^j}
$
satisfies
%\newline
${\mathcal P}(1)=-{\mathcal R}(\frac{1}{\lambda},0)=0$,
and, therefore, ${\mathcal P}(w)=(w -1){\mathcal Q}(w)$,
where ${\mathcal Q}$ is a polynomial of degree $k-1$. Now
relations \eqref{inverse2} and \eqref{block} show that
both $\tilde{\Psi}_1$ and $\tilde{\tilde{\Psi}}_2$ have
poles of order at most one at $w =1$, and
relation \eqref{vanishing_residue} implies that
$\Psi_1$ is holomorphic at $w=1$. 
%We are done.
\end{proof}

\section{Line in the spectrum}
\label{S:line-in-spectrum}

Now, suppose that, as in the previous section, $A_1$ and
$A_2$ are self-adjoint, that $\lambda\neq 0$ is an isolated
spectral point of $A_1$, and that 
$\sigma_p( A_1,A_2)\cap\Delta(\frac{1}{\lambda},0)$
coincides, up to multiplicity, with a line segment
$
\{
(x,y)\in \Delta_\rho(\frac{1}{\lambda},0): \
\lambda x+ ay=1
\}
$
where $a\ne 0$.
Passing to $A_1/\lambda$ and $A_2/a$ if necessary,
we may assume that $\lambda=a=1$, that is,
\[
%\left( 
\{x+y=1\} \cap \Delta_\rho(1,0)
%\right)
= \sigma_p(A_1,A_2)\cap \Delta_\rho(1,0)
\]
up to multiplicity. 
Coming back to relation \eqref{polynom}, here we have
$k=1, \ r_1^0=r_1^1=1, \ \mu_1=1, x_1=1$. Let us write
down the functions $\Psi_1$ and $\Psi_2$ in this
particular case. Equations \eqref{condition2},
\eqref{inverse2}, and \eqref{block} imply
\begin{eqnarray}
\Psi_1(w)=
\left(
\frac{1}{w-1}P_1 - \sum_{m=0}^\infty (w -1)^mT^{m+1}
\right)
\left(
(w -1) \left[ \frac{1}{w -1}{A_2P_1} \right.
\right.
\nonumber \\
\left.
\left.
-\sum_{m=0}^\infty (w-1)^mA_2T^{m+1}
\right]
- I
\right)
\label{res1}
\end{eqnarray}
\begin{eqnarray}
\Psi_2(w) = \Psi_1(w)\left[A_2(w I-A_1)^{-1}\right]
\nonumber \\
=
\left(\frac{1}{w -1}P_1-\sum_{m=0}^\infty (w -1)^mT^{m+1}\right)
\left(
(w -1)
\left[
\frac{1}{w -1}A_2P_1
\right.
\right.
\label{res2}\\
\left.
\left.
- \sum_{m=0}^\infty (w-1)^m A_2T^{m+1}
\right]
- I
\right)
\left[
\frac{1}{w -1} A_2P_1-\sum_{m=0}^\infty (w-1)^mA_2T^{m+1}
\right].
\nonumber
\end{eqnarray}
It follows from \eqref{vanishing_residue},
\eqref{res1}, and \eqref{res2} that
\begin{eqnarray}
{\mathcal R}es_{\lambda=1}(\Psi_1) = P_1A_2P_1 - P_1=0
\label{residue1}\\
{\mathcal R}es_{\lambda=1}(\Psi_2) =
P_1A_2TA_2P_1- P_1(A_2P_1-I)A_2T
\nonumber \\
- T(A_2P_1-I)A_2P_1 = 0 \label{residue2}
\end{eqnarray}
The last two equations imply
\begin{equation}\label{residue3}
P_1A_2TA_2P_1=0.
\end{equation}

%\vspace{.2cm}

\begin{remark}
%{\bf Remark}
Coming back to the beginning of this section,
suppose that %\newline
\[
%\left(
\{\lambda x+ay=1\} 
\cap \Delta_\rho\bigl(\frac{1}{\lambda},0\bigr)
%\right)
%\subset 
= \sigma_p(A_1,A_2) 
\cap \Delta_\rho\bigl(\frac{1}{\lambda},0\bigr)
\]
up to multiplicity. 
Then the operators $A_1/\lambda$ and $A_2/a$
satisfy \eqref{residue1} and \eqref{residue2}. Since the
projections $P_j$ and the operator $T$ for $A_1$ and
$A_1/a$ are the same, we obtain
\begin{equation}\label{residue11}
P_1A_2P_1=aP_1.
\end{equation}
Equation \eqref{residue3} stays the same.
\end{remark}

Now we use \eqref{residue3} to establish necessary and
sufficient conditions for a common eigenvector in the case
when at least one of the operators $A_1, A_2$ is invertible.

\begin{lemma}\label{invertible1}
Let $A_1,A_2$ be self-adjoint, $1$ be an isolated spectral
point of $A_1$, and assume that $A_1$ is invertible.
If there is $\rho >0$ such that
\[
%\left(
\{x+y=1\} \cap \Delta_\rho(1,0)
%\right)
= \sigma_p(A_1,A_2)\cap \Delta_\rho(1,0) 
\]
up to multiplicity, 
then the following are equivalent:
\begin{itemize}
\item[(1)]
$A_1$ and $A_2$ have an $n$-dimensional common
eigensubspace, where $n=Rank(P_1)$, and the whole line
$\{ x+y=1\}$ is in $\sigma_p(A_1,A_2)$;

\item[(2)]
There is $\rho^\prime >0$ such that the  line segment
$\{ x+y=1\}\cap \Delta_{\rho^\prime}(1,0)$ %is in
agrees with 
$\sigma_p(A_1^{-1},A_2)\cap \Delta_{\rho^\prime}(1,0)$ 
up to multiplicity;

\item[(3)]
There is $\rho^{\prime \prime}$ such that the plane
segment $\{x+y+z=1\}\cap \Delta_{\rho^{\prime \prime}}(1,1,0)$
%is in 
agrees with 
$\sigma_p(A_1,A_1^{-1},A_2)\cap \Delta_{\rho''}(1,1,0)$ 
up to multiplicity.
\end{itemize}
\end{lemma}

\begin{proof}
The implications
$(1)\Longrightarrow (2), \ (1)\Longrightarrow (3)$, and
$(3)\Longrightarrow (2)$  are obvious. Thus,
it suffices to prove $(2)\Longrightarrow (1)$.

Suppose that (2) holds. Let $L_1$ be the eigensubspace of
$A_1$ with eigenvalue one. Choose an orthonormal basis of
$L_1$: \ $e_1,\dots$. Equation \eqref{operatorT} implies
that in our case for every $\xi \in H$
$$
T(A_1)\xi = T\xi =
\int_{\sigma(A_1)\setminus \{1\}}\frac{1}{z-1}dE(z)(\xi).
$$
Therefore, using that $A_2$ is self-adjoint we have for
every $j$
\begin{eqnarray}
P_1A_2TA_2P_1e_j =
\sum_m
\left(
\int_{\sigma(A_1)\setminus \{1\}}\frac{1}{z-1}
\langle dE(z)A_2e_j, A_2e_m\rangle
\right)
e_m
\nonumber \\
=
\sum_m
\left(
\int_{\sigma(A_1)\setminus \{1\}}\frac{1}{z-1}
\langle dE(z)A_2e_j, dE(z)A_2e_m\rangle
\right)
e_m.
\nonumber
\end{eqnarray}
Equation \eqref{residue3} implies that for every pair $j,m$
$$
\int_{\sigma(A_1)\setminus \{1\}}\frac{1}{z-1}
\langle dE(z)A_2e_j, dE(z)A_2e_m \rangle = 0.
$$
In particular, when $j=m$ we obtain
\begin{equation}\label{common1}
\int_{\sigma(A_1)\setminus \{1\}}\frac{1}{z-1}
\left\Vert dE(z)A_2e_j\right\Vert^2=0.
\end{equation}
We now apply all preceding considerations to the pair
$(A_1^{-1}, A_2)$. First we observe that
$$
A_1^{-1} =
P_1+\int_{\sigma(A_1)\setminus \{1\}} \frac{1}{z}dE(z).
$$
Hence,
$$
P_1A_2\tilde{T}A_2P_1=0,
$$
where
\begin{equation}\label{T_of_inverse}
\tilde{T} = T(A_1^{-1}) =
\int_{\sigma(A_1)\setminus \{1\}} \frac{z}{1-z}dE(z).
\end{equation}
In a similar way the last two relations yield
\begin{equation}\label{common2}
\int_{\sigma(A_1)\setminus \{1\}}\frac{z}{1-z}
\left\Vert dE(z)A_2e_j\right\Vert^2 = 0.
\end{equation}
Adding \eqref{common1} and \eqref{common2} we obtain
$$
\int_{\sigma(A_1)\setminus \{1\}}
\left\Vert dE(z)A_2e_j\right\Vert^2=0.
$$
This means that $A_2e_j \in L_1$ for every $j$. Thus,
$L_1$ is invariant under $A_2$. Since the restriction
of $A_1$ to $L_1$ is the identity operator,
the joint spectrum of $A_2|_{L_1}$ and the identity of
$L_1$ contains a cone with vertex at $(1,0)$ that
contains every line of the family
$\{x+\frac{y}{a}=1: \ a\in \sigma(A_2|_{L_1})\}$, and,
of course, this cone lies in $\sigma_p(A_1,A_2)$. Since
the intersection of $\sigma_p(A_1,A_2)$ with a
neighborhood of $(1,0)$ is a line segment, we conclude
that the spectrum of $A_2|_{L_1}$ consists of a single
point, and, since $A_2$ is self-adjoint, this means that
$L_1$ is an eigenspace for $A_2$. 
%We are done.
\end{proof}

Since eigenvectors of an operator $A$ and its scalar
multiple are the same, the following result  is a
straightforward corollary to Lemma~\ref{invertible1}.

\begin{lemma}\label{invertible2}
Let $A_1,A_2$ be self-adjoint, $\lambda \neq 0$ be an
isolated point of $\sigma(A_1)$, and $A_1$ be invertible.
If there exist $a\ne 0$ and  $\rho>0$ such that, 
up to multiplicity, 
$
\{\lambda x+ay=1\} \cap \Delta_\rho(\frac{1}{\lambda},0) =
\sigma_p(A_1,A_2)\cap \Delta_\rho(1,0),
$
then the following are equivalent:
\begin{itemize}
\item[(1)]
$A_1$ and $A_2$ have a common eigensubspace of
dimension equal to the rank of $P_1$;

\item[(2)]
there is $\rho^\prime$ such that, 
up to multiplicity, 
$
\sigma_p(A^{-1}_1,A_2)\cap \Delta_{\rho^\prime}(\lambda,0) =
\{ \frac{x}{\lambda}+ay=1 \} \cap
\Delta_{\rho^\prime}(\lambda,0);
$

\item[(3)]
There is $\rho^{\prime \prime}>0$ such that, 
up to multiplicity, 
$
\{ \lambda x+\frac{y}{\lambda}+az = 1\} \cap
\Delta_{\rho^{\prime \prime}}(\lambda, \frac{1}{\lambda},0) =
\sigma_p(A_1,A^{-1}_1,A_2)\cap
\Delta_{\rho^{\prime \prime}}(\lambda, \frac{1}{\lambda},0).
$
\end{itemize}
\end{lemma}

We will use the result of Lemma~\ref{invertible2}  to give
a necessary and sufficient condition for a common
eigenvector for an arbitrary pair of self-adjoint operators.
To this end, for any self-adjoint operator $A$
we consider the following
family of perturbations: %of $A$
\begin{equation}\label{perturb1}
A(\epsilon,\lambda) =
(1+\epsilon)A-\lambda\epsilon I, \
\epsilon \in \R ,  \  \epsilon \neq -1.
\end{equation}

\begin{remark}\label{perturb-by-identity}
It is easily seen that for every $\epsilon,\lambda \in \R$ 
the operator 
$A(\epsilon, \lambda)$ is self-adjoint.  
Furthermore, if $\lambda$ is an
isolated spectral point of $A$, then it is an isolated
spectral point of $A(\epsilon, \lambda)$ for every
$\epsilon \neq -1$; and %that
the line segment
$
\{\lambda x+ay=1\}\cap \Delta_\rho(\frac{1}{\lambda},0)
$
is in
$\sigma_p(A_1,A_2)$ if and only if it is in
$\sigma_p(A_1(\epsilon,\lambda),A_2(\epsilon,a))$.
It is also straightforward that the eigensubspace of
$A(\epsilon,\lambda)$ corresponding to  eigenvalue
$\lambda$ is either empty
or is the same for all $\epsilon\neq -1$.
We further remark that if $\lambda\neq 0$, then there
exists $\epsilon$ such that $A(\epsilon,\lambda)$ is
invertible. Indeed, the spectral mapping theorem, cf.
\cite[Chapter 7, Section 3, Thm. 11]{DS}, implies that
$
\sigma(A(\epsilon, \lambda)) =
(1+\epsilon)\sigma(A)- \lambda \epsilon.
$
Thus $0\in \sigma(A(\epsilon, \lambda))$ if and only if
$\frac{\lambda \epsilon}{1+\epsilon}\in \sigma(A)$.
Since $\lambda$ is an isolated point of $\sigma(A)$, if
$\epsilon \in \R$ and $|\epsilon|$ is big enough, zero
is not in the spectrum of $A(\epsilon, \lambda)$, that is
$A(\epsilon, \lambda)$ is invertible.
\end{remark}

Since by making a linear change of
coordinates (which  amounts to replacing $A_2$ by
$A_2+\delta A_1$ for a sufficiently small real $\delta$)
we can always reduce to the case when $a\ne 0$,
the following result is an immediate corollary to
Lemma~\ref{invertible2}.

\begin{theorem}\label{line_in_spectrum}
Let $A_1,A_2$ be self-adjoint operators on $H$, let
$\lambda$ be an isolated spectral point of $\sigma(A_1)$,
and suppose that in some neighborhood
$\Delta_\rho(\frac{1}{\lambda},0)$ of $(\frac{1}{\lambda},0)$
%\ $\Delta_\rho(\frac{1}{\lambda},0)$,
the joint spectrum $\sigma_p(A_1,A_2)$ coincides 
up to multiplicity with a
line segment
$
\{
(x,y)\in \Delta_\rho(\frac{1}{\lambda},0): \ \lambda x+ ay=1
\}.
$
%with $a\ne 0$.
The following are equivalent:
\begin{itemize}
\item[(1)]
The eigenspace of $A_1$ corresponding to the eigenvalue
$\lambda$ is also an eigensubspace for $A_2$;

\item[(2)]
There exist $\epsilon \in \R, \ \epsilon \neq -1$ and
$\rho^\prime >0$ such that $A_1(\epsilon,\lambda)$ is
invertible and the line segment
$
\{
(x,y)\in \Delta_{\rho^\prime}(\lambda,0): \
\frac{x}{\lambda}+ay=1
\}
$
coincides up to multiplicity with
$
\sigma_p((A_1(\epsilon,\lambda))^{-1},A_2(\epsilon,a))\cap
\Delta_{\rho^\prime}(\lambda,0);
$

\item[(3)]
There exist $\epsilon \in \R, \ \epsilon \neq -1$ and
$\rho^{\prime \prime}>0$ such that $A_1(\epsilon,\lambda)$
is invertible and the plane segment
$
\{
(x,y,z)\in
\Delta_{\rho^{\prime \prime}}(\frac{1}{\lambda},\lambda,0): \
\lambda x+\frac{1}{\lambda}y+az=1
\}
$
coincides with
$
\sigma_p(A_1(\epsilon,\lambda), A_1(\epsilon,\lambda)^{-1},
                         A_2(\epsilon,a))\cap
\Delta_{\rho^{\prime \prime}}(\frac{1}{\lambda},\lambda,0)
$
up to multiplicity. 
\end{itemize}
\end{theorem}

As a direct corollary to  Theorem~\ref{line_in_spectrum}
we obtain the following result for an $n$-tuple of
self-adjoint operators.

\begin{theorem}\label{line_in_spectrum_tuple}
Let $A_1,\dots,A_n$ be self-adjoint, let $\lambda \neq 0$ be
an isolated point of $\sigma(A_1)$, and suppose there exists
$\rho>0$ such that, up to multiplicity,
\begin{align*}
\{
\lambda x_1+a_2x_2+\dots+a_nx_n=1\} & \cap
\Delta_\rho(1/\lambda,0,\dots,0) \\
= \sigma_p(A_1,\dots,A_n) & \cap
\Delta_\rho(1/\lambda,0,\dots,0).
\end{align*}
The following are equivalent:
\begin{itemize}
\item[(1)]
The eigensubspace of $A_1$ corresponding to eigenvalue
$\lambda$ is an eigensubspace for each of the operators
$A_2,\dots,A_n$;

\item[(2)]
There exist an $\epsilon \in \R, \ \epsilon \neq 1$ and
$\rho^{\prime}>0$ such that $A_1(\epsilon,\lambda)$ is
invertible and %the plane segment
\begin{align*}
\{
(1/\lambda) x_1+a_2x_2+\dots+a_nx_n=1
\} & \cap
\Delta_\rho(\lambda,0,\dots,0) \\
=
\sigma_p( A_1(\epsilon,\lambda)^{-1}, A_2(\epsilon,a_2),\dots,
          A_n(\epsilon,a_n)) & \cap
\Delta_{\rho^\prime}(\lambda,0,\dots ,0),
\end{align*}
up to multiplicity. %\\
\end{itemize}
\end{theorem}

\begin{proof}
Obviously (1) implies (2).

Suppose that (2) holds. Since the line segments
$
\{
(x_1,x_j)\in \Delta_\rho(\frac{1}{\lambda},0): \
\lambda x_1+a_jx_j = 1
\}
$
and
$
\{
(x_1,x_j)\in \Delta_\rho(\lambda,0): \
\frac{1}{\lambda}x_1+a_jx_j = 1
\}
$
coincide with
$
\sigma_p(A_1(\epsilon,\lambda),A_j(\epsilon,a_j))\cap
\Delta_\rho(\frac{1}{\lambda},0)
$
and
$
\sigma_p(A_1(\epsilon,\lambda)^{-1},A_j(\epsilon,a_j)) \cap
\Delta_{\rho^\prime}(\lambda,0)
$
respectively for all $j=2,\dots, n$, it follows from
Theorem~\ref{line_in_spectrum} that the eigenspace of
$A_1$ that corresponds to the eigenvalue $\lambda$ is
an eigenspace of $A_j$ for all $j=2,\dots,n$, and (1) holds.
\end{proof}

\section{Spectral algebraic curves, exterior powers,
and common reducing subspaces}
\label{S:exterior-powers}

It is clear that if operators $A_1$ and $A_2$ have a
common reducing subspace of dimension  $n$, the joint
spectrum $\sigma_p(A_1,A_2)$
contains an algebraic curve of order $n$. Our example
from the Introduction shows that in general the converse
is not true. In this section we use results of the
previous section to establish necessary and sufficient
conditions under which the presence of an algebraic curve
in the joint spectrum implies the existence of a common
reducing subspace. As our conditions are expressed  in terms of
joint spectra of exterior products,
we begin by recalling some basic facts about exterior
products of Hilbert spaces. For more details we refer the
reader to \cite[Chapter V.1]{Te} and
\cite[Chapter X.7]{NFBK}.

For any $n\ge 1$ the $n$th tensor power
$\otimes^n H$ of a Hilbert space $H$
has inner product given by
\[
\langle
x_1\otimes\dots\otimes x_n, y_1\otimes\dots\otimes y_n
\rangle =
\langle x_1, y_1\rangle \dots \langle x_n,y_n\rangle
\]
The $n$th \emph{exterior power} $\bigwedge^n H$ of $H$ is
defined as the quotient of $\otimes^n H$ modulo the subspace
generated by all elements of the form
\[
x_1\otimes\dots\otimes x\otimes x\otimes\dots\otimes x_n
\]
The image of a simple tensor $x_1\otimes\dots\otimes x_n$ in
$\bigwedge^nH$ is denoted by $x_1\wedge\dots\wedge x_n$. We
consider $\bigwedge^nH$ as a subspace of
$\otimes^n H$ via the canonical ``antisymmetrizing map''
\[
x_1\wedge\dots\wedge x_n \longmapsto \frac{1}{\sqrt{n!}}
\sum_{\sigma\in S_n}\sign(\sigma)
x_{\sigma(1)}\otimes\dots\otimes x_{\sigma(n)}
\]
In particular, $\bigwedge^nH$ inherits via this map a
Hilbert space structure from $\otimes^n H$, and it is
straightforward to compute that
its inner product satisfies
\[
\langle
x_1\wedge\dots\wedge x_n, \ y_1\wedge\dots\wedge y_n
\rangle =
\det
\left[
\begin{array}{ccc}
\langle x_1, y_1\rangle & \dots & \langle x_1, y_n\rangle \\
\vdots                 &        & \vdots               \\
\langle x_n, y_1\rangle & \dots & \langle x_n, y_n\rangle
\end{array}
\right].
\]
Therefore if $\{e_1,\dots, e_n, \dots\}$ is an
orthonormal basis
of $H$ we obtain that the set
$
\{
e_{i_1}\wedge\dots\wedge e_{i_n} \mid
1\le i_1<\dots < i_n
\}
$
is an orthonormal basis of $\bigwedge^nH$.

When $A$ is a linear operator on $H$ it induces a
linear operator
$\wedge^n A$ on $\bigwedge^n H$ via the formula
$
[\wedge^n A](v_1\wedge\dots\wedge v_n)=
Av_1\wedge \dots\wedge Av_n.
$
Thus $\wedge^n A$ is just the restriction of $\otimes^n A$
to $\bigwedge^n H$ considered as a subspace of $\otimes^n H$.
It is immediate that if $A$ is of finite rank or
self-adjoint then so is $\wedge^n A$; and
when $A$ is a bounded we get
$\left\Vert\wedge^n A\right\Vert\le
\left\Vert A\right\Vert^n$.
In particular, compactness of $A$ 
implies compactness of $\wedge^n A$.

When $A$ is self-adjoint and compact
and $\lambda_1,\dots $ are the
eigenvalues of $A$, with $e_1,\dots $ being a
corresponding eigenbasis,
then the orthogonal set
$
\{e_{i_1}\wedge\dots\wedge e_{i_n} \mid  1\leq i_1<\dots<i_n \}
$
is an eigenbasis with $\{ \lambda_{i_1}\dots\lambda_{i_n} \}$
as the corresponding multiset of eigenvalues for
the compact self-adjoint %operator
$\wedge^n A$.

\begin{definition}
Let $A$ be a self-adjoint operator. We say that a finite
multiset $L=\{\lambda_1, \dots, \lambda_n\}$ is a
\emph{spectral multiset} for $A$ if each $\lambda_i$ is
an isolated point in the spectrum of $A$
of finite multiplicity, and the number of times
it occurs in $L$ is at most its multiplicity
as an eigenvalue of $A$. We say that a spectral multiset
for $A$ is \emph{generic}
if $\lambda=\lambda_1\dots\lambda_n$ is isolated and
of multiplicity $1$
in the spectrum of $\wedge^n A$.
\end{definition}

\begin{remark}
A straightforward consequence of the definition is
that if $L=\{\lambda_1,\dots,\lambda_n\}$
is a generic spectral multiset for the
self-adjoint operator $A$, then each eigenvalue
$\lambda_i$ of $A$ has multiplicity equal to the number
of times it occurs in $L$.
\end{remark}

\begin{remark}\label{R:invertible-generic}
Suppose $A$ is a self-adjoint operator
with countable spectrum, and
$L=\{\lambda_1,\dots,\lambda_n\}$ is a spectral multiset
for $A$ such that each eigenvalue
$\lambda_i$ of $A$ has multiplicity equal to the number
of times it occurs in $L$.
Then there is an open subset $U\subset\R$ such
that $\R\setminus U$ is countable, and such that for every
$\delta\in U$ the operator $A+\delta I$
is invertible and has
$L+\delta=\{\lambda_1+\delta,\dots,\lambda_n+\delta\}$
as a generic spectral multiset. Indeed,
since $A+\delta I$ is %always
invertible for $\delta\in\R\setminus-\sigma(A)$,
and since the spectrum of $\wedge^n(A+\delta)$ is a subset of
the spectrum of $\otimes^n (A+\delta)$,
it suffices to show that
for every $\delta$ in the set
\begin{multline*}
\bigl\{
\delta\in\R \ \big| \
(\lambda_1+\delta)\dots(\lambda_n+\delta) \ne
(\mu_1 +\delta)\dots(\mu_n+\delta) \\
\text{ for $\mu_i\in\sigma(A)$ such that
$\{\mu_1,\dots,\mu_n\}\ne L$ as multisets}
\bigr\}
\end{multline*}
the point
$(\lambda_1+\delta)\dots(\lambda_n+\delta)$ is isolated
in the spectrum of $\otimes^n (A+\delta I)$. But, as
the spectrum of a tensor
product of operators is the product of their spectra,
cf. \cite{BP},
and each $\lambda_i+\delta$ is isolated in the spectrum
of $A+\delta I$, this follows from the compactness of
spectra by a standard argument.
\end{remark}

\begin{comment}
For any operator $A$ we set
\[
\widetilde A(\delta, a, n)=
[\wedge^n A](\delta, a) =
(1+\delta)[\wedge^n A] - \delta a[\wedge^n I]
\]
\end{comment}

The assumptions in the following theorem describe what
we consider to be a ``general position'' setting.

\begin{theorem}\label{T:generic}
Let $A_1, A_2$ be self-adjoint operators with $A_1$
invertible.
Consider a generic spectral multiset
$L=\{\lambda_1,\dots, \lambda_n\}$ %be a multiset
for $A_1$, and
let $\lambda = \lambda_1\dots\lambda_n$.
Suppose also that
for some $a\ne 0$ and $\rho>0$
the line segments
\[
\begin{aligned}
\{\lambda x + ay =1\} & \cap\Delta_{\rho}(1/\lambda, 0)
\text{ and } \\
\{(1/\lambda)x+ ay=1\}& \cap\Delta_{\rho}(\lambda, 0)
\end{aligned}
\]
are inside
$\sigma_p(\wedge^n A_1, \wedge^nA_2)$ and %in
$
\sigma_p(\wedge^n A_1^{-1}, \wedge^nA_2)
$
respectively.

Then the eigenspace of $A_1$ corresponding to the
eigenvalues
$\lambda_1,\dots,\lambda_n$ is invariant under $A_2$.
\end{theorem}

\begin{proof}
Since $\lambda$ is isolated and simple in the spectrum of
$\wedge^n A_1$ it follows that the joint spectrum
$\sigma_p(\wedge^n A_1,\wedge^n A_2)$ is nonsingular at
the point $(1/\lambda, 0)$, in particular this point
belongs to no component
other than the line $\lambda x + a y = 1$.
By Lemma~\ref{invertible2} the operators $\wedge^n A_1$ and
$\wedge^n A_2$ have a common unit eigenvector $v$
(of eigenvalue $\lambda$ for
$\wedge^n A_1$ and eigenvalue $a$ for $\wedge^n A_2$).
Since $L$ is generic $v$ must be
of the form $v=e_1\wedge\dots\wedge e_n$,
where each $e_i$ is a unit eigenvector for $A_1$ of
eigenvalue $\lambda_i$.
Now we show that $span(e_1,\dots, e_n)$
is invariant under $A_2$. Indeed,  let $e$
be any other eigenvector for
$A_1$, and consider the column vector
\[
w=
\left[
\begin{array}{c}
\langle A_2e_1, e\rangle \\
\vdots                   \\
\langle A_2e_n, e\rangle
\end{array}
\right].
\]
For $i=1,\dots n$ set
$v_i=
e_1\wedge\dots\wedge e_{i-1}\wedge e
   \wedge e_{i+1}\wedge\dots\wedge e_n$
and note that each $v_i$ is orthogonal to $v$.
As $\wedge^nA_2(v)= a v$,
it follows that
\[
a =
\langle\wedge^n A_2(v), v\rangle =
\langle
A_2e_1\wedge\dots\wedge A_2e_n,
e_1\wedge\dots\wedge e_n
\rangle
\]
and therefore $\det Y = a \ne 0$,
where $Y$ is the matrix
\[
Y =
\left[
\begin{array}{ccc}
\langle A_2e_1, e_1\rangle & \dots
                           & \langle A_2e_1, e_n\rangle \\
\vdots                 &
                       & \vdots               \\
\langle A_2e_n, e_1\rangle & \dots
                           & \langle A_2e_n, e_n\rangle
\end{array}
\right]
.
\]
In particular the linear system of equations
\[
Y
\left[
\begin{array}{c}
x_1    \\
\vdots \\
x_n
\end{array}
\right]
= w
\]
has a unique solution given,
according to Cramer's Rule, by the formula
\[
x_i=\frac{\det Y_i(w)}{a}
\]
for each $i=1,\dots,n$, where $Y_i(w)$
is obtained by replacing
with $w$ the $i$th column of the matrix $Y$. But since
\[
\det Y_i(w) =
\langle
\wedge^n A_2(v), e_1\wedge\dots\wedge e\wedge\dots\wedge e_n
\rangle
=
a
\langle
v, v_i
\rangle = 0
\]
we see that $x_i=0$ for each $i$, and therefore $w=0$.
\end{proof}

Recall that %we say that
a self-adjoint operator $A$ on an infinite dimensional
separable Hilbert space $H$ belongs
to the class $\mathcal E(H)$ when $A=K + aI$ with $K$ a
compact self-adjoint operator on $H$. In this case every
point in
$\sigma(A)\setminus a$  is isolated of finite multiplicity,
and the point $a$ is either an accumulation point, or
isolated of infinite multiplicity.

For operators in the class $\mathcal E(H)$
we are now ready to address the question of when the
presence of an algebraic curve in the proper joint
spectrum indicates the existence of a common reducing
subspace. Consider two self-adjoint operators
$A=K_1 + aI$ and $B= K_2 + bI$ in $\mathcal E(H)$,
and suppose that $\Gamma$ is an algebraic
curve of degree $k$ which is a union of components
%(not necessarily irreducible)
of the proper joint spectrum
$\sigma_p(A, B)$. Note that the line $\{ ax + by = 1\}$
is always in $\sigma_p(A, B)$ and therefore carries no
information about common reducing subspaces. 
We will refer to this line as 
the \emph{accumulation line} of the joint spectrum. 
Thus, without
loss of generality we can assume that it is not %contained
a reduced component of $\Gamma$. Therefore,
by making a linear change of coordinates
if necessary (which amounts to replacing $A$ and $B$ by
appropriate linear combinations
of $A$ and $B$ with real coefficients, hence does not
affect the presence and degrees of algebraic curves or
common reducing subspaces and their dimensions)
we may also assume
that $\Gamma$ intersects the $x$-axis in $k$ points
(counted with multiplicities)
$
1/\lambda_1,\dots,1/\lambda_k
$
and the $y$-axis in $k$ points
$
1/\mu_1, \dots, 1/\mu_k,
$
and that each point
$(1/\lambda_i, 0)$ belongs only to components
of $\sigma_p(A,B)$ contained in $\Gamma$; %, and
in particular, the multiset
$
L=\{\lambda_1,\dots,\lambda_k \}
$
is a spectral mutiset for $A$ and
each eigenvalue
$\lambda_i$ of $A$ has multiplicity equal to the number
of times it occurs in $L$.
Therefore, by Remark~\ref{R:invertible-generic},
if needed, we can make an additional suitable linear
fractional change of coordinates of the form
\[
u = \frac{x}{1+\delta x},  \qquad
v = \frac{y}{1+\delta x}
\]
(which amounts to replacing $A$ by $A + \delta I$)
and also assume that $A$ is invertible, and
that $L$ is a generic spectral multiset for $A$.
Thus, for operators in the class $\mathcal E(H)$
we can always reduce the search for a common invariant
subspace to the
``general position'' case considered in the following
theorem, which is one of the main results in this paper.

\begin{theorem}\label{spectral_curves}
Let $A=K_1+ aI$ and $B=K_2+bI$ be self-adjoint operators
in the class $\mathcal E(H)$, with $A$ invertible.
Let $\Gamma$ be an algebraic curve of degree $k$ which is a
union of components %(irreducible but not necessarily reduced)
of the proper joint spectrum $\sigma_p(A,B)$,
and which does not have the accumulation line 
$\{ax + by = 1 \}$ as a reduced component.
Suppose that the $x$-axis (resp. the $y$-axis) intersects
$\Gamma$ in the $k$ points (counted with multiplicity)
$1/\lambda_1,\dots, 1/\lambda_k$
(resp. $1/\mu_1,\dots,1/\mu_k$)
such that each point $(1/\lambda_i,0)$ belongs only to
components of $\sigma_p(A,B)$ contained in $\Gamma$,
and the multiset
$L=\{\lambda_1,\dots,\lambda_k\}$ is a generic spectral
multiset for $A$.
Set $\lambda=\lambda_1\dots\lambda_k$ and
$\mu=\mu_1\dots\mu_k$.
The following are equivalent:
\begin{itemize}
\item[(1)]
The eigenspace for $A$ corresponding to
$\lambda_1,\dots,\lambda_k$ is invariant for $B$.

\item[(2)]
There exists
$\rho>0$ such that
the line segments
\[
\begin{aligned}
\{\lambda x + \mu y =1\} & \cap\Delta_{\rho}(1/\lambda, 0)
\text{ and } \\
\{(1/\lambda)x+ \mu y=1\}& \cap\Delta_{\rho}(\lambda, 0)
\end{aligned}
\]
are %inside
contained %, respectively,
in %the proper joint spectra
$
\sigma_p\bigl(
\wedge^k A, \wedge^k B
\bigr)
$
and %in
$
\sigma_p\bigl(
\wedge^k A^{-1}, \wedge^k B
\bigr),
$
respectively.

\item[(3)]
The lines
\begin{equation*}
\{\lambda x + \mu y =1\} %& \cap\Delta_{\rho}(1/\lambda, 0)
\quad\text{ and }\quad %\\
\{(1/\lambda)x+ \mu y=1\} %& \cap\Delta_{\rho}(\lambda, 0)
\end{equation*}
are %inside
contained in 
$
\sigma_p\bigl(
\wedge^k A, \wedge^k B
\bigr)
$
and %in
$
\sigma_p\bigl(
\wedge^k A^{-1}, \wedge^k B
\bigr),
$
respectively.
\end{itemize}
\end{theorem}

\begin{proof}
The implications $(1)\Rightarrow (3)\Rightarrow (2)$
are straightforward, %
and $(2)$ implies $(1)$ % .
by Theorem~\ref{T:generic}.
\end{proof}

The case when $H$ is a finite-dimensional Hilbert space is of
classical interest in algebraic geometry. In that setting
every linear operator is bounded and of finite rank, hence
we can take $\mathcal E(H)$ to be the space of all linear
operators on $H$. Now
the statement of our last theorem can be
somewhat simplified and
slightly expanded as follows.

Suppose that $\dim_\C H = N$, and let $H^*$ denote the 
space dual to $H$. Then %for any $0\le m\le N$
we have the canonical isomorphism 
\[
\wedge^k H^*\longrightarrow\wedge^{N-k}H\otimes\wedge^N H^*, 
\]
%through 
which, together with the inner product induced  
isomorphism 
$\wedge^k H \rightarrow \wedge^k H^*$,   
%\quad\text{ and }\quad 
allows us to consider in a natural way  
%consider 
$\wedge^{N-k}A$ as acting on
the space $\wedge^k H$ for any linear operator  
%whenever the self-adjoint operator
$A$ on $H$. In particular, we can consider the proper
joint spectrum $\sigma_p(\wedge^{N-k}A, \wedge^k B)$.
When $A$ is self-adjoint,  
%Also, it is 
a standard exercise in multilinear algebra shows that
the Sylvester expansion formula for the determinant
transforms into the equality 
$(\wedge^{N-k}A)(\wedge^k A)=(\det A)I$. Thus, 
for an invertible self-adjoint $A$ we
have $\wedge^{N-k}A = \det(A) \wedge^k A^{-1}$, 
and therefore  
%Now the proof of
Theorem~\ref{spectral_curves} implies immediately
the following result.

\begin{theorem}\label{spectral_curves_matrices}
With assumptions and notation as in
Theorem~\ref{spectral_curves}, suppose in addition that
$a=b=0$ and $\dim_\C H = N$.
The following are equivalent:
\begin{itemize}
\item[(1)]
The eigenspace for $A$ corresponding to
$\lambda_1,\dots,\lambda_k$ is invariant for $B$.

\item[(2)]
For some $\rho>0$ the line segments
\[
\begin{aligned}
\{\lambda x + \mu y =1\} & \cap\Delta_{\rho}(1/\lambda, 0)
\text{ and } \\
\{(1/\lambda)x+ \mu y=1\}& \cap\Delta_{\rho}(\lambda, 0)
\end{aligned}
\]
are %inside
contained in
%inside the proper joint spectra
$
\sigma_p\bigl(
\wedge^k A, \wedge^k B
\bigr)
$
and %in
$
\sigma_p\bigl(
\wedge^k A^{-1}, \wedge^k B
\bigr),
$
respectively.

\item[(3)]
For some $\rho>0$ the line segments
\[
\begin{aligned}
\{\lambda x + \mu y =1\} & \cap\Delta_{\rho}(1/\lambda, 0)
\text{ and } \\
\{(\det A/\lambda)x+ \mu y=1\}&
\cap\Delta_{\rho}(\lambda/\det A, 0)
\end{aligned}
\]
are %inside
contained in
%inside the proper joint spectra
$
\sigma_p\bigl(
\wedge^k A, \wedge^k B
\bigr)
$
and %in
$
\sigma_p\bigl(
\wedge^{N-k} A, \wedge^k B
\bigr),
$
respectively.

\item[(4)]
The lines
%segments
\begin{equation*}
\{\lambda x + \mu y =1\} %& \cap\Delta_{\rho}(1/\lambda, 0)
\quad\text{ and }\quad %\\
\{(\det A/\lambda)x+ \mu y=1\}
%& \cap\Delta_{\rho}(\lambda, 0)
\end{equation*}
are %inside
contained
%respectively,
in
$
\sigma_p\bigl(
\wedge^k A, \wedge^k B
\bigr)
$
and %in
$
\sigma_p\bigl(
\wedge^{N-k} A, \wedge^k B
\bigr),
$
respectively.
\end{itemize}
\end{theorem}

Finally, we note that Theorem~C from the Introduction 
is obtained from the result above in the same way that 
Theorem~\ref{line_in_spectrum_tuple} is obtained from 
Theorem~\ref{line_in_spectrum}.

\section{Example: unit circle in the spectrum}
\label{S:example}

Here we present an example of a setting where the existence of
a common reducing subspace is established without exterior products
but rather in a way similar to
Theorem~\ref{line_in_spectrum}.

We start with an elementary result from \cite{CSZ}
about spectra of linear transforms.
Let
\[
{\bf C}=
\left[
\begin{array}{ll}
   c_{11} & c_{12} \\
  c_{21} & c_{22}
\end{array}
\right ]
\]
be a complex-valued matrix. Write
\begin{equation}\label{transform}
%\label{lin}
B_1= c_{11}A_1+c_{12}A_2, \ B_2=c_{21}A_1+c_{22}A_2.
\end{equation}

\begin{proposition}
$\sigma_p(A_1,A_2)\supset  {\bf C}^T \sigma_p(B_1,B_2)$.
\end{proposition}

The proof is straightforward.

\begin{corollary}\label{spectral_transform}
If {\bf C} is invertible, then
\begin{equation}\label{spectransform}
\sigma_p(A_1,A_2)={\bf C}^T \sigma_p(B_1,B_2).
\end{equation}
\end{corollary}

The following result shows that under some
conditions the presence of
a circle in the joint spectrum implies the
existence of a common reducing subspace.

\begin{theorem}\label{circle}
Let $A_1,A_2$ be self-adjoint and suppose that $A_1$ is
invertible, $\left\Vert A_2\right\Vert =1$, and both
$\pm 1$ be isolated spectral points for $A_1$ of finte
multiplicites $k_1$ and $l_1$ and be eigenvalues of
$A_2$ of finite multiplicities $k_2$ and $l_2$ respectively
with $k_2+l_2\leq k_1+l_1$. If there is $\rho >0$ such that
%\begin{eqnarray}
\begin{equation}\label{isolated_circle}
\begin{split}
\{x^2+y^2=1\}\cap \Delta_\rho(z,0) &=
\sigma_p(A_1,A_2)\cap \Delta_\rho(z,0) %\nonumber
\\
&=
\sigma_p(A_1^{-1},A_2)\cap \Delta_\rho(z,0); \qquad
z=\pm 1.
\end{split}
\end{equation}
%\end{eqnarray}
Then
\begin{itemize}
\item[(1)]
$k_1=l_1=k_2=l_2$;

\item[(2)]
$A_1$ and $A_2$ have a common $2n$-dimensional reducing
subspace $L$, where $n=k_1=l_1=k_2=l_2$;

\item[(3)]
the pair of restrictions of $A_1|_L$ and $A_2|_L$ is
unitary equivalent to a pair of $2n\times 2n$ involutions
$C_1$ and $C_2$, where
$$
C_1=
\left[
\begin{array}{cc}
                  I_n & 0_n \\
                  0_n & -I_n
\end{array}
\right], \
C_2=
\left[
\begin{array}{cc}
                            0_n & D_n \\
                        D_n^\ast & 0_n
\end{array}
\right],
$$
$0_n$ and $I_n$ are zero and identity $n\times n$ matrices
and $D$ is a unitary $n\times n$ matrix;

\item[(4)]
The group generated by $C_1$ and $C_2$ represents
Coxeter's group $BC_2$.
\end{itemize}
\end{theorem}

\begin{proof}
Let  $e_1,...e_{k_1}$ and $e_{k_1+1},...e_{k_1+l_1}$ be a pair
of  orthonormal  eigenbases for the eigenspaces of $A_1$
with eigenvalues $1$ and $-1$ respectively and
$\xi_1,...,\xi_{k_2}$ and $\xi_{k_2+1},...,\xi_{k_2+l_2}$ be a
similar pair of orthonormal eigenbases for $A_2$.
Write down the spectral decomposition of $A_1$:
$$
A_1=P_1-P_2+\int_{\sigma(A_1)\setminus \{ \pm 1\} } zdE(z),
$$
where $P_1$ and $P_2$ are the original projections on
the spaces $span\{ e_1,...,e_{k_1}\}$ and
$span \{ e_{k_1+1},...,e_{k_1+l_1}\}$ respectively. Applied
to this particular situation equation
\eqref{vanishing_residue} for $m=1$ and $m=2$ gives
\begin{eqnarray}
P_1A_2P_1=0 \label{circle_order_one} \\
P_1A_2TA_2P_1= - \frac{1}{2}P_1 \label{circle_order_2}
\end{eqnarray}

Equation \eqref{circle_order_2} implies that for
every $j=1,...,k_1$ we have
\begin{equation}\label{squares_for A_1}
\sum_{m=k_1+1}^{k_1+l_1}
\frac{|\langle A_2e_j,e_m\rangle |^2}{2} +
\int_{\sigma(A_1)\setminus \{\pm 1\} }
\frac{|\langle A_2e_j, dE(z)\rangle|^2}{1-z}=\frac{1}{2}.
\end{equation}
The same computation applied to the unit circle in
$\sigma_p(A_1^{-1},A_2)$ yields
$$
P_1A_2T(A_1^{-1})A_2P_1=-\frac{1}{2}P_1,
$$
and equation \eqref{T_of_inverse} now gives
\begin{equation}\label{squares_for_inverse}
-\sum_{m=k_1+1}^{k_1+l_1}
\frac{|\langle A_2e_j,e_m\rangle |^2}{2} +
\int_{\sigma(A_1)\setminus \{\pm 1\} }
\frac{z|\langle A_2e_j, dE(z)\rangle|^2}{1-z}=-\frac{1}{2}.
\end{equation}
Subtracting \eqref{squares_for_inverse} from
\eqref{squares_for A_1} we obtain
\begin{equation}\label{sum_of_squares}
\sum_{m=k_1+1}^{k_1+l_1} |\langle A_2e_j,e_m\rangle|^2 +
\int_{\sigma(A_1)\setminus \{\pm 1\} }
|\langle A_2e_j, dE(z)\rangle|^2=1.
\end{equation}
Since \eqref{circle_order_one} means $P_1(A_2e_j)=0$,
the relation \eqref{sum_of_squares} implies
\begin{equation}\label{norm_one}
\left\Vert A_2e_j\right\Vert =1.
\end{equation}
Since $\Vert A_2\Vert = 1$, it follows from
\eqref{norm_one} that for every $1\leq j\leq k_1$, \
$e_j\in span \{ \xi_1,... \xi_{k_2+l_2}\}$. A similar
analysis applied to the decomposition near the spectral
point $(-1,0)$   shows that for each
$k_1 + 1\leq j\leq k_1+l_1$, \
$e_j\in span \{ \xi_1,...\xi_{k_2+l_2}\}$. Thus,
$
span\{ e_1,...,e_{k_1+l_1} \} \subset
span\{ \xi_1,...\xi_{k_2+l_2}\}
$
and, since $k_2+l_2\leq k_1+l_1$, we have $k_1+l_1=k_2+l_2$
and
$
L =
span\{ \xi_1, ...\xi_{k_2+l_2}\} =
span \{ e_1,...e_{k_1+l_1}\}
$
is a common reducing subspace of $A_1$ and $A_2$. We also
remark that the restrictions of $A_1$ and $A_2$ to $L$ are
involutions.

To show that $k_1=l_1=k_2=l_2$ observe that condition
\eqref{isolated_circle}, spectral continuity and
Corollary~\ref{spectral_transform} imply that for every
real $\theta$ close to zero the operator
$A_\theta = (\cos\theta) A_1 + (\sin\theta) A_2$ has $L$ as
an invariant subspace, is self-adjoint and the spectrum
of the restriction of $A_\theta$ to
$L$ consists of $\pm 1$. Hence, $A_\theta|_L$ is an
involution, and we have
\begin{multline*}
I = A_\theta^2=
(\cos^2\theta) A_1^2 + (\sin^2\theta) A_2^2 +
(\cos\theta) (\sin\theta) (A_1A_2+A_2A_1) %\nonumber
\\
= I + (\cos\theta)(\sin\theta) (A_1A_2+A_2A_1). %\nonumber
\end{multline*}
If $\sin 2\theta \neq 0$, the last relation implies
\begin{equation}\label{involution}
A_1A_2=-A_2A_1.
\end{equation}
In the basis $e_1,...,e_{k_1+l_1}$ the restriction of
$A_1$ to $L$ has the form
$$
\left[
\begin{array}{cc}
          I_{k_1} & 0_{k_1\times l_1} \\
           0_{l_1\times k_1} & I_{l_1}
\end{array}
\right].
$$
Since  $A_2$ is self-adjoint, it follows from
\eqref{involution} that the restriction of $A_2$
to $L$ in the same basis has the form
\begin{equation}\label{involution1}
\left[
\begin{array}{cc}
   0_{k_1} & D_{k_1 \times l_1} \\
  D_{k_1\times l_1}^\ast & 0_{l_1}
\end{array}
\right].
\end{equation}
If $k_1\neq l_1$, the determinant of the matrix
\eqref{involution1} is equal to zero contradicting the
fact that it is an involution. Thus, $k_1=l_1= 1/2(k_1+l_1)$.
Also, the trace of matrix \eqref{involution1} is equal to
zero, which implies that $k_2=l_2=k_1=l_1$, and both (2)
and (3) are proved

Since the restrictions of $A_1$ and $A_2$ to $L$ are
involutions, to prove (4) it suffices to verify that
$$
(A_1|_L A_2|_L)^4=I.
$$
This can be done by a direct computation. 
%We are done.
\end{proof}

%\vspace{.2cm}

\begin{remark}
%{\bf Remark}.
As we saw above, the joint spectrum of a pair of compact
operators is an analytic set, and, of course, away from the 
accumulation line, the same is
true for operators in  ${\mathcal E}(H)$. 
In this setting it is easy
to show using spectral continuity and the fact that the
set of singularities of an analytic set has higher
codimension, that the spectrum is a divisor, that is the
multiplicity is the same at every pair of points in the
joint spectrum belonging to the same irreducible component
of the spectrum where this component is smooth. Thus, in
such setting the statement (1) of Theorem~\ref{circle} does
not require any special proof.
\end{remark}

%\vspace{.2cm}

\section{Spectral continuity and
common ``almost eigenvectors''}
\label{S:spectral-continuity}

Spectral continuity is  well-known in the classical
spectral theory, see \cite{CM}. In our case it implies
that if $A_{1,n}\to A_1$ and $A_{2,n}\to A_2$ in operator
norm topology as $n\to \infty$, \ then
$\sigma_p(A_{1,n},A_{2,n})$
converges to $\sigma_p(A_1,A_2)$ in Hausdorff topology
uniformly on compact subsets of $\C^2$.
In particular, this implies that if $A_1$ and $A_2$ have
a common eigenvector, $A_{1,n}$ and $A_{2,n}$ have a
common ``almost eigenvector'' (we define it below) and
$\sigma_p(A_{1,n},A_{2,n})$ contains an
irreducible component that converges to a line in
Hausdorff topology uniformly on compacts as $n\to \infty$.
In this section we prove  results that establish the
converse: under certain natural assumptions local
closeness of $\sigma_p(A_1,A_2)$ to a line  implies
existence of a common almost eigenvector.

%\vspace{.2cm}

\begin{definition}
%\noindent {\bf Definition} {\it
We say that a non-zero vector $\xi$ is an
$\epsilon$-eigenvector (almost eigenvector)
of an operator $A$  if there exists
$\lambda \in \C$ such that
$
\Vert A\xi -\lambda \xi \Vert \leq \epsilon
\Vert \xi \Vert.
$
\end{definition}

%\vspace{.2cm}

%\noindent
Since the distance from $A\xi$ to the line
$\{ \lambda \xi : \ \lambda \in \C \}$ is equal to
$
\Vert
A\xi - \frac{\langle A\xi,\xi \rangle}{\Vert \xi\Vert^2}\xi
\Vert,
$
we come to an equivalent definition of an
$\epsilon$-eigenvector:
\emph{$\xi$ is an $\epsilon$-eigenvector of $A$} if
\begin{equation}\label{epsilon-eigenvector}
\left\Vert
A\xi - \frac{\langle A\xi,\xi\rangle}{\Vert \xi\Vert^2}\xi
\right\Vert
\leq \epsilon \Vert \xi  \Vert.
\end{equation}
It immediately follows from \eqref{epsilon-eigenvector}
that $\xi$ is an eigenvector of $A$ in the traditional
sense, if and only if it is an $\epsilon$-eigenvector
for all $\epsilon >0$. More generally, if
$\lambda \in \sigma(A)$, then for every $\epsilon >0$
there exists an $\epsilon$-eigenvector $\xi$ such that
$\Vert A\xi-\lambda \xi \Vert \leq\epsilon\Vert \xi \Vert$.

Of course, every vector $\xi$ is an $\epsilon$-eigenvector
with the appropriate choice of $\epsilon$ to be  equal to
the lefthand side of \eqref{epsilon-eigenvector}, but
this is quite meaningless. The notion of an
$\epsilon$-eigenvector is meaningful when $\epsilon$
is small. In this case being an $\epsilon$-eigenvector
means that $A\xi$ lies in a small aperture cone that has
the line $\{ \lambda \xi : \ \lambda \in \C \}$ as the
symmetry axis.

%\noindent
Our next result is a generalization of
Theorem~\ref{line_in_spectrum} to the case of common almost
eigenvectors for compact operators. Let $\Gamma$ be an
analytic curve that passes through $(x,y)\in \C^2$ and let
$\rho >0$. We will use the following notation:
\[
\Gamma_\rho(x,y)=
\Gamma \cap
\Delta_\rho(x,y).
\]
If  $\epsilon$ is close to zero and $A_1(\epsilon)$ and
$A_2(\epsilon)$ are close to $A_1$ and $A_2$ respectively,
then  spectral continuity implies that locally
$\sigma_p(A_1,A_2)$ is close to
$\sigma_p(A_1(\epsilon), A_2(\epsilon))$. For this reason
in the next theorem without loss of generality we may
assume that $A_1$ is invertible.  To simplify the notation
we will also use rescaling, if necessary, so that the
point $(1,0)$ is in the joint spectrum, of $A_1$ and $A_2$
and $A_1$ is invertible. Finally, %let us write
recall that the
operator $A$ belongs to the class ${\mathcal E}(H)$,
if it is represented as $A=K+\alpha I$ where $K$ is
compact and $\alpha \in \R$.

\begin{theorem}\label{close to line}
Let $A_1, A_2\in {\mathcal E}(H)$  
such  that $1\in \sigma(A_1)$, 
and, therefore, the point $(1,0)$ belongs to
$\sigma_p(A_1,A_2)$  and to $\sigma_p(A_1^{-1}, A_2)$.
Suppose that $1$ is not an accumulation point of
$\sigma(A_1)$ and $(1,0)$ is not
a singular point of either $\sigma_p(A_1,A_2)$, or
$\sigma_p(A_1^{-1},A_2)$. Let  $\sigma_p(A_1,A_2)$  and
$\sigma_p(A_1^{-1},A_2)$ near $(1,0 )$ be zeros of analytic
functions $f_1(x,y)$ and $f_2(x,y)$ respectively.
If there exist
$0<\rho <1$ and  $0<\epsilon \ll \rho$ such that
\begin{itemize}
\item[(1)]
$
d=1-(1-\rho)\Vert A_1\Vert -
\epsilon\sqrt{2}\Vert A_2\Vert > 0;
$
%\\

\item[(2)]
the Hausdorff distances from $\sigma_p(A_1,A_2)_\rho (1,0)$
and $\sigma_p(A_1,A_2)_\rho(1,0)$ to the line
$\{x+\beta y=1\}$ are less than $\epsilon $
($\beta$ is a real number); %\\

\item [(3)]
$
\frac{\partial f_j}{\partial x} +
\beta\frac{\partial f_j}{\partial y}\neq 0
$
in $\Delta_\rho(1,0)$, \ $j=1,2$;
\end{itemize}
then $A_1$ and $A_2$ have a common $\delta$-eigenvector,
where $\delta=D\sqrt{\epsilon}$, and $D$ is a constant
independent of $\beta$.
\end{theorem}

\begin{proof}
First we observe that conditions (1) and (2) imply that
$|\beta|$ has an upper bound expressed in terms of
$\rho, \epsilon$ and the norms of $A_1$ and $A_2$. Indeed,
suppose that $|\beta|>1$. Without loss of generality we
may assume that $\beta >0$. It is shown below that there
is a point $(1-\rho, \tau)\in \sigma_p(A_1,A_2)$.
The distance from this point to $\{x+\beta y=1\}$ is equal
to $\frac{|\beta \tau -\rho|}{\sqrt{1+\beta^2}}$.
Condition (2) implies
$$
\frac{\rho}{\beta}-
\epsilon\sqrt{1+\frac{1}{\beta^2}}\leq \tau \leq
\frac{\rho}{\beta}+\epsilon \sqrt{1+\frac{1}{\beta^2}},
$$
so that
$$
|\tau|\leq \frac{\rho}{\beta}+\epsilon \sqrt{2}.
$$
Since the operator $(1-\rho)A_1+\tau A_2-I$ is not
invertible, we have
\begin{eqnarray}
1\leq \Vert (1-\rho)A_1+\tau A_2\Vert
 \leq (1-\rho)\Vert A_1\Vert +
|\tau| \Vert A_2\Vert  \nonumber \\
\leq (1-\rho)\Vert A_1\Vert +
(\frac{\rho}{\beta} +\epsilon \sqrt{2})\Vert A_2\Vert.
\nonumber
\end{eqnarray}
This implies
\begin{equation}\label{beta_estimate}
\beta \leq
\frac{\rho \Vert A_2\Vert}
     {1-(1-\rho)\Vert A_1\Vert -
      \epsilon\sqrt{2}\Vert A_2\Vert } =
\frac{\rho \Vert A_2\Vert}{d}.
\end{equation}
% observe that replacing $\epsilon$ with $2\epsilon$ and
% $A_1$ and $A_2$ with close operators we may assume that
% both joint spectra are %algebraic curves. Indeed, we can
Now we approximate the compact parts of $A_1$ and $A_2$ by
finite rank operators with simple spectra (that is every
non-zero eigenvector has multiplicity one) resulting in
operators $\tilde{A}_1$ and $\tilde{A}_2$. We can find
$\tau \in \R$ as close to zero as we want, such that
$\hat{A}_1=\tilde{A}_1+\tau I$ is invertible. Since
$\hat{A}_1$ and $\tilde{A}_2$ are close to $A_1$ and $A_2$
respectively, the spectral continuity
implies that $\sigma_p(\hat{A}_1,\tilde{A}_2)$ is close to
$\sigma_p(A_1,A_2)$ in the bidisk
$\{ |x-1|\leq \rho, \ |y|\leq \rho \}$. Also if $\hat{A}_1$
and $\tilde{A}_2 $ are close enough to $A_1$ and $A_2$
respectively, then $\delta$-eigenvectors for $\hat{A}_1$
and $\tilde{A}_2$ are $2\delta$-eigenvectors for $A_1$ and
$A_2$. It is clear that $\sigma_p(\hat{A}_1,\tilde{A}_2)$
and $\sigma_p(\hat{A}_1^{-1},\tilde{A}_2)$ are algebraic
sets. Finally, the points of intersection of
$\sigma_p(\hat{A}_1, \tilde{A}_2)$ and
$\sigma_p(\hat{A}_1^{-1},\tilde{A}_2)$ with the $x$-axis
that are close to one, are regular points, and the
distances between the line $\{ x+\beta y=1\}$ and
$\sigma_p(\hat{A}_1,\tilde{A}_2)$ and
$\sigma_p(\hat{A}_1^{-1},\tilde{A}_2)$ is less than
$2\epsilon$. We denote the polynomials that determine
$\sigma_p(\hat{A}_1,\tilde{A}_2)$ and
$\sigma_p(\hat{A}_1^{-1},\tilde{A}_2)$ by ${\mathcal R}(x,y)$
and ${\mathcal S}(x,y)$. As the rank of the approximating
operators increases and $\tau$ approaches zero, the
polynomials ${\mathcal R}$ and ${\mathcal S}$ approach
$f_1$ and $f_2$ respectively. It follows from the direct
expression  of  the defining function given by equation
\eqref{determine}. Thus,  the condition $(2)$ holds for
these polynomials with $\epsilon_1=2\epsilon$. Again
rescaling with a coefficient close to one we may assume
that both $\sigma_p(\hat{A}_1,\tilde{A}_2)$ and
$\sigma_p(\hat{A}_1^{-1},\tilde{A}_2)$ pass through $(1,0)$.

%\noindent
Next we note that since $\epsilon \ll \rho$, the orthogonal
projection in $\C^2$ onto $\{x+\beta y=1\}$ of each  curve
$\sigma_p(\hat{A}_1,\tilde{A}_2)$ and
$\sigma_p(\hat{A}_1^{-1},\tilde{A}_2)$ contains the disk of
radius $\frac{\rho}{\sqrt{2}}$ centered at $(1,0)$.
Indeed, again we may assume $\beta \geq 0$.
Let us change the coordinates to
\begin{equation}\label{u-v_coordinates}
u =
\frac{x-1}{\sqrt{1+\beta^2}} +
\frac{\beta y}{\sqrt{1+\beta^2}}, \
v =
\frac{\beta (x-1)}{\sqrt{1+\beta^2}} -
\frac{y}{\sqrt{1+\beta^2}}.
\end{equation}
It is easily seen that the the bidisk
$\{|x-1|\leq \rho, \ |y|\leq \rho \}$ contains the bidisk
$
\Delta_\rho(\beta) =
\{
|u|\leq \frac{\rho \sqrt{1+\beta^2}}{1+\beta}, \
|v|\leq \frac{\rho \sqrt{1+\beta^2}}{1+\beta}
\}.
$
Since
$\frac{\sqrt{1+\beta^2}}{1+\beta} \geq \frac{1}{\sqrt{2}}$,
the bidisk
$
\Delta_\rho =
\{
|u| \leq \frac{\rho}{\sqrt{2}}, \
|v| \leq \frac{\rho}{\sqrt{2}}
\}
$
is in $\{ |x-1|\leq \rho, \ |y|\leq \rho \}$. In the
$(u,v)$-coordinates $\sigma_p(\hat{A}_1,\tilde{A}_2)$ and
$\sigma_p(\hat{A}_1^{-1},\tilde{A}_2)$ are zeros of the
polynomials
$
\tilde{{\mathcal R}}(u,v)={\mathcal R}
(\frac{u+\beta v}{\sqrt{1+\beta^2}}+1,
 \frac{\beta u - v}{\sqrt{1+\beta^2}})
$
and
$
\tilde{{\mathcal S}}(u,v) = {\mathcal S}
(\frac{u+\beta v}{\sqrt{1+\beta^2}} + 1,
 \frac{\beta u - v}{\sqrt{1+\beta^2}})
$
respectively. Suppose that $(u_0,v_0)\in \Delta_\rho$  and
$\tilde{{\mathcal R}}(u_0,v_0)=0$, that is
$(u_0,v_0)\in \sigma_p(\hat{ A}_1, \tilde{A}_2)$.  The
distance from this point to the line
$\{ x+\beta y= 1\}=\{ u=0\}$ is equal  to $|u_0|$. Hence,
$|u_0|\leq 2\epsilon$. Consider the functions
$\phi_v(u)=\tilde{{\mathcal R}}(u,v)$. Since
$\phi_{v_0}(u_0)=0$, Hurwitz's theorem (see, for example,
\cite[page 231]{G}) implies that if $v_1$ is close to
$v_0$, \ $\phi_{v_1}$ has  zero $u_1$ close to $u_0$.
Since $|u_0|\leq 2\epsilon \ll \frac{\rho}{\sqrt{2}}$,
we conclude that $|u_1|<\frac{\rho}{\sqrt{2}}$, and,
therefore, $(u_1,v_1)\in \Delta_\rho$, and $(0,v_1)$
belongs to the projection of
$\sigma_p(\hat{A}_1,\tilde{A}_2)$ onto $\{ u=0\}$. A
similar argument applied to
$\sigma_p(\hat{A}_1^{-1},\tilde{A}_2)$ finishes the proof
of the claim.

%\noindent
We now return to relations \eqref{recurrence} and
\eqref{psi1} to express residues  of $\Psi_1$ and $\Psi_2$
explicitedly in terms of derivatives of the determining
polynomial. It is easy to check that in our case these
relations for $\sigma_p(\hat{A}_1,\tilde{A}_2)$
yield
\begin{eqnarray}
{\mathcal R}es \ \Psi_1(\lambda)\left|_{\lambda=1}\right.
=
\frac{\partial {\mathcal R}}{\partial x}
\bigg|_{(1,0)}
%\right.
P_1\tilde{A}_2P_1 -
\frac{\partial{\mathcal R}}{\partial y}
\bigg|_{(1,0)}
%\right.
P_1 =0  \label{small1} \\
{\mathcal R}es \ \Psi_2(\lambda)
\left|_{\lambda=1}\right.
=
-\frac{\partial{\mathcal R}}{\partial x}
\bigg|_{(1,0)}
%\right.
P_1\tilde{A}_2T\tilde{A}_2P_1 -
\frac{1}{2}
\left(\frac{\partial^2{\mathcal R}}{\partial x^2}
\left(\frac{\frac{\partial {\mathcal R}}{\partial y}}
           {\frac{\partial {\mathcal R}}{\partial x}}
\right)^2
\right. \label{small2} \\
\left. -2
\frac{\partial ^2{\mathcal R}}{\partial x\partial y}
\frac{\frac{\partial{\mathcal R}}{\partial y}}
     {\frac{\partial{\mathcal R}}{\partial x}}
+
\frac{\partial^2{\mathcal R}}{\partial y^2}
\right)
\Bigg|_{(1,0)}
%\right.
P_1=0, \nonumber
\end{eqnarray}
where, as in Section~3, $P_1$ is the orthogonal projection
on the eigenspace of $\hat{A}_1$ corresponding to the
eigenvalue one, and $T$ is defined by \eqref{operatorT}.

Equation~\eqref{small1} together with condition (3) of
this theorem imply that
$
\frac{\partial{\mathcal R}}{\partial x}
\left|_{(1,0)}\right. \neq 0,
$
and, hence, this derivative does not vanish in a
neighborhood of $(1,0)$. By the implicit function theorem
the relation ${\mathcal R}(x,y)=0$ determines $x$ as an
analytic function of $y$ in a neighborhood of $(1,0)$.
In terms of this function $x(y)$ equations \eqref{small1}
and \eqref{small2} can be written as
\begin{eqnarray}
 P_1\tilde{A}_2P_1=-x^\prime (0)P_1,  \label{small3} \\
P_1\tilde{A}_2T\tilde{A}_2P_1 =
-\frac{x^{\prime \prime}(0)}{2} P_1. \label{small4}
\end{eqnarray}
As it was done before, we denote by $e_1,e_2,...$ an
orthonormal eigenbasis for $\hat{A}_1$ with
$
\hat{A}_1(e_1)=e_1,
\hat{A}_1(e_j)=\mu_je_j, \ \mu_j\neq 1, \ j=2,...
$
and $P_j $ being the orthogonal projection onto a
subspace spanned by $e_j$. In this basis relation
\eqref{small4} can be written as
\begin{equation}\label{small4_1}
\sum_{j=2}^\infty
\frac{|\langle \tilde{A}_2e_1,e_j\rangle|^2}{\mu_j-1} =
-\frac{x^{\prime \prime }(0)}{2}.
\end{equation}

Our next step is to show that $x^{\prime \prime} (0)$ is small.
Let us once again pass to the coordinates
\eqref{u-v_coordinates}). We have
$
\frac{\partial \tilde{{\mathcal R}}}{\partial u} =
\frac{1}{\sqrt{1+\beta^2}}
\frac{\partial {\mathcal R}}{\partial x} +
\frac{\beta}{\sqrt{1+\beta^2}}
\frac{\partial {\mathcal R}}{\partial y}\neq 0
$
for every
$
(u,v)\in
\{
|u|\leq \frac{\rho}{\sqrt{2}}, \
|v|\leq \frac{\rho}{\sqrt{2}}
\}.
$
Applying the implicit function theorem we see that
equation $\tilde{{\mathcal R}}(u,v)=0$ determines $u$ as
an analytic  function of $v$ in a neighborhood of every
point $v\in \{ |v|\leq \frac{\rho}{\sqrt{2}}\}$. Since
this function is globally continuous in
$\{|v|\leq \frac{\rho}{\sqrt{2}} \}$, by the monodromy
theorem, see \cite[p.161]{G}, $u$ is holomorphic in the
whole disk $\{ |v|\leq \frac{\rho}{\sqrt{2}} \}$. The above
argument showed that\ $|u(v)|\leq 2\epsilon$ for every
$v\in \{|v|\leq \frac{\rho}{\sqrt{2}} \}.$ Now the Cauchy
theorem implies that
\begin{eqnarray}
|u^\prime(0)|\leq \frac{4\epsilon}{\rho} \label{small5}, \\
|u^{\prime \prime}(0)|\leq \frac{4\sqrt{2}\epsilon}{\rho ^2}.
\label{small6}
\end{eqnarray}
A straightforward computation shows that
\begin{equation}\label{small7}
\frac{d^2 x}{dy^2} =
\frac{(1+\beta^2)^{3/2}\frac{d^2 u}{dv^2}}
     {(1-\beta\frac{du}{dv})^3}.
\end{equation}
Equations \eqref{beta_estimate}, \eqref{small5},
\eqref{small6}, and \eqref{small7} yield
\begin{equation}\label{small8}
|x^{\prime \prime}(0)|\leq C\epsilon
\end{equation}
where
\begin{equation}\label{C}
C =
\frac{4\sqrt{2}\left(1+
               \left(\frac{\rho \Vert A_2\Vert}{d}\right)^2
               \right)^{3/2}\rho}
     {\left(\rho -\frac{4\rho \Vert A_2\Vert}{d}\epsilon
      \right)^3}
\end{equation}
is a constant independent of $\beta$. Now equations
\eqref{small4_1} and \eqref{small8} give
\begin{equation}\label{small9}
\left|
\sum_{j=2}^\infty
\frac{|\langle \tilde{A}_2e_1, e_j\rangle|^2}{\mu_j-1}
\right|
\leq \frac{C}{2}\epsilon.
\end{equation}
Applying a  similar argument to the pair
$\hat{A}_1^{-1},\tilde{A}_2$ and using \eqref{T_of_inverse}
we obtain
\begin{equation}\label{small10}
\left|
\sum_{j=2}^\infty
\frac{\mu_j |\langle \tilde{A}_2e_1,e_j\rangle|^2}
     {\mu_j-1}
\right|
\leq \frac{C}{2}\epsilon .
\end{equation}
Equations \eqref{small9} and \eqref{small10} result in
$$
\sum_{j=2}^\infty
|\langle \tilde{A}_2e_1,e_j\rangle |^2
\leq C\epsilon.
$$
Set $\lambda = \langle \tilde{A}_2e_1,e_1\rangle$. The
last relation implies
\begin{equation}\label{small11}
\Vert \tilde{A}_2e_1 -\lambda e_1\Vert \leq
\sqrt{C\epsilon},
\end{equation}
which means that $e_1$ is $\tilde{\delta}$-eigenvector of
$\tilde{A}_2$, where $\tilde{\delta}=\sqrt{C\epsilon}$.
Thus, $e_1$ is a common $\tilde{\delta}$-eigenvector of
$\hat{A}_1$ and $\tilde{A}_2$. Therefore, as it was
mention above, $e_1$ is a $\delta$-eigenvector of $A_1$
and $A_2$ with $\delta =2\tilde{\delta}$. We are done.
\end{proof}

%\vspace{.2cm}

It was mentioned in the proof of Theorem~\ref{close to line}
that the polynomial ${\mathcal R}$ determining the spectrum of
$\hat{A}_1$ and $\tilde{A}_2$ converges uniformly on compacts
in a neighborhood of $(1,0)$ to the function $f$ determining
$\sigma_p(A_1,A_2)$ as the finite rank approximations of
compact parts converge to those of $A_1$ and $A_2$. Thus,
we have the following corollary to the proof of
Theorem~\ref{close to line}.

\begin{corollary}\label{first_moments}
Suppose that $A_1,A_2\in {\mathcal E}(H)$, with 
$(1,0)\in \sigma_p(A_1,A_2)$.
Suppose further that $f(x,y)$ is an analytic function
that determines $\sigma_p(A_1,A_2)$ near $(1,0)$ and
$\frac{\partial f}{\partial x} \left|_{(1,0)}\right. \neq 0$.
If $P_1$ is the orthogonal projection onto
the eigensubspace of $A_1$ corresponding to $\lambda=1$ and
$T$ is given by \eqref{operatorT}, then
\begin{eqnarray}
P_1A_2P_1 = -x^\prime (0)P_1,
\label{first_moment_implicit} \\
P_1A_2TA_2P_1 = -\frac{x^{\prime \prime}(0)}{2} P_1,
\label{second_moment_implicit}
\end{eqnarray}
where $x(y)$ is the implicit function near $y=0$ determined
by the equations $f(x,y)=0, \ x(0)=1$.
\end{corollary}

If in Theorem~\ref{close to line} the norm of $A_2$ is equal
to $|\beta|$, then no condition on $\sigma_p(A_1^{-1},A_2)$ is
necessary for the existence of a common almost eigenvector.

\begin{theorem}\label{almost}
Let $A_1,A_2$ be compact and $\Vert A_2\Vert =|\beta|$.
Suppose that $\alpha\neq 0$ and $(1/\alpha,0)$ belongs to
$\sigma_p(A_1,A_2)$. If there exist $\rho>0$ and
$0<\epsilon \ll \rho$ such that
\begin{itemize}
\item[(1)]
the Hausdorff distance from
$\sigma(A_1,A_2)_\rho (1/\alpha,0)$ to the line
$\{ \alpha x+\beta y=1\}$ does not exceed $\epsilon$;

\item[(2)]
$
\alpha \frac{\partial f}{\partial x} +
\beta \frac{\partial f}{\partial y} \neq 0
$
in the bidisk
$\{ |x-1/\alpha|\leq \rho, \ |y|\leq \rho \}$, where $f$
is an analytic function that determines $\sigma_p(A_1,A_2)$,
\end{itemize}
then there is an eigenvector of $A_1$ that is
$
(
\sqrt{\frac{8|\beta|(1+\beta^2)}{\rho-4\beta\epsilon}}
\sqrt{\epsilon}
)
$-eigenvector of $A_2$.
\end{theorem}

\begin{proof}
As we mentioned before, we can replace   $A_1$ with
$A_1/\alpha$, so that $\alpha =1$. Also, as in
Theorem~\ref{close to line} using an arbitrary small
perturbation we may assume that  eigenvalue $\lambda=1$
of $A_1$  has multiplicity one. Condition $(2)$ implies
that in the bidisk $\{|x-1|\leq \rho, \ |y|\leq \rho \}$
the joint spectrum $\sigma_p( A_1,A_2)$ is nonsingular and,
therefore, is a smooth analytic curve $\Gamma$. Using
condition $(1)$ and the argument with passing to the
coordinates \eqref{u-v_coordinates} similar to the one
that was used in the proof of Theorem~\ref{close to line}
and the fact that
$
\frac{dx}{dy}=
\frac{\frac{du}{dv}+\beta}{\beta \frac{du}{dv} -1}
$
we show that
$
|-x^\prime (0) -\beta|\leq
\frac{ 4(1+\beta^2)\epsilon}{\rho-4\beta \epsilon}.
$
Thus, if $e_1$ is a unit eigenvector of $A_1$ with
eigenvalue $\lambda=1$, the relation
\eqref{first_moment_implicit} implies
$$
|\langle A_2e_1,e_1\rangle|=|x^\prime (0)|\geq
|\beta|-\frac{4(1+\beta^2)\epsilon}{\rho - 4\beta\epsilon}.
$$
Hence,
\begin{eqnarray}
\Vert A_2e_1 - \langle A_2e_1, e_1\rangle e_1 \Vert^2 =
\Vert A_2e_1\Vert^2-|\langle A_2e_1,e_1\rangle|^2
\nonumber \\
\leq \beta^2 -
\left(
|\beta|-\frac{4(1+\beta^2)\epsilon}
             {\rho-4\beta\epsilon}
\right)^2
\leq
\frac{8|\beta|(1+\beta^2)\epsilon}{\rho-4\beta\epsilon}.
\label{estimate}
\end{eqnarray}
We are done.
\end{proof}

%\vspace{.2cm}

\begin{remark}
%{\bf Remark}
The condition $\Vert A_2\Vert =|\beta|$ can obviously be
replaces with
$\Bigl| \Vert A_2\Vert -|\beta| \Bigr| <\delta$. In this
case there is a common
$
\sqrt{
2\beta \delta + \delta^2 +
\frac{8\beta(1+\beta^2)\epsilon}{\rho-4\beta \epsilon}
}
$
-eigenvector.
\end{remark}

\section{Norm estimates for the
commutant of a pair of matrices}
\label{S:norm-estimates}

%\vspace{.2cm}

Under the assumptions of Theorem~\ref{almost} we will now
define a new operator close to $A_2$ that has a common
eigenvector with $A_1$. Let as above $e_1$ be an eigenvector
of $A_1$ with $\lambda=1$. Write
$$
\hat{A}_2=P_1A_2 P_1+(I-P_1)A_2(I-P_1).
$$
Of course, $e_1$ is a common eigenvector of $A_1$ and
$\hat{A}_2$.

Let $\xi$ be a unit vector orthogonal to $e_1$, that is
$\Vert \xi \Vert = 1, \ (I-P_1)\xi=\xi$. We have
$$
\Vert A_2\xi- \hat{A}_2\xi \Vert =
\Vert P_1A_2 \xi\Vert =
|\langle A_2e_1,\xi\rangle | =
|
\langle( A_2e_1 - \langle A_2e_1,e_1 \rangle e_1), \xi
\rangle
|
\leq C\sqrt{\epsilon},
$$
where
$
C =
\sqrt{
\frac{8\Vert A_2\Vert(1+\Vert A_2\Vert^2)}
     {\rho-4\Vert A_2\Vert \epsilon}
}.
$
For $\zeta=ce_1+\sqrt{1-|c|^2}\xi$ with
$\Vert \xi\Vert =1, \ (I-P_1)\xi=\xi$ the last
relation yields
\begin{equation}\label{distances}
\Vert (A_2-\hat{A}_2)\zeta \Vert \leq |c|
\Vert A_2e_1-\langle A_2e_1,e_1\rangle e_1)\Vert +
\sqrt{1-|c|^2}\Vert A_2\xi-\hat{A}_2\xi\Vert \leq
\sqrt{2}C\sqrt{\epsilon},
\end{equation}
and, therefore,
$$
\Vert A_2-\hat{A}_2\Vert \leq \sqrt{2}C\sqrt{\epsilon}.
$$
This gives us the following estimate for the norm of the
commutant $[A_1,A_2]$:
\begin{eqnarray}
\Vert [A_1,A_2]\Vert \leq \Vert [A_1,(A_2-\hat{A}_2]\Vert +
\Vert [A_1,\hat{A}_2]\Vert \nonumber \\
\leq \sqrt{2}C\sqrt{\epsilon}
\Vert A_1\Vert + \Vert [A_1^{(1)},A_2^{(1)}]\Vert,
\label{commutant1}
\end{eqnarray}
where
$A_1^{(1)}=(I-P_1)A_1(I-P_1), \ A_2^{(1)}=(I-P_1)A_2(I-P_1)$
are the compressions of $A_1$and $A_2$ to the
orthocomplement to $e_1$.

\begin{remark}
%{\bf Remark}
If the the point $(\frac{1}{\alpha},0)$ is not a singular
point of the proper joint spectrum of $A_1$ and $A_2$ with 
$\Vert A_2\Vert =|\beta|$, and
$\sigma_p(A_1,A_2)_\rho(\frac{1}{\alpha},0)$ is at less
than $\epsilon$ Hausdorff distance from the line
$\{ \alpha x+\beta y=1\}$, the inequality \eqref{commutant1}
still holds.
This follows from the fact that the pair
$(A_1/\alpha, A_2)$ satisfies the conditions of
Theorem~\ref{almost}.
\end{remark}

Now we will use the relation \eqref{commutant1} to
estimate the norm of the commutant of a pair of
self-adjoint $N\times N$ matrices in terms of the
Hausdorff distance from the joint spectrum to a set
of lines that imitates a joint spectrum of a pair of
commuting matrices. Since our result is stable with
respect to small perturbations, in the next theorem
we will assume that both matrices have simple spectra
and the absolute values of their eigenvalues are different.
Since the commutant  of $A_1$ and $A_2$ is the same as
the commutant of $A_1+\alpha I$ and $A_2+\beta I$ for every
$\alpha, \ \beta$, we may assume that $A_1$ and $A_2$ are
invertible, that is all their eigenvalues are nontrivial.

Let $f(z)$ be analytic in the closed disk
$\overline{\Delta_\rho(a)}=\{ |z-a|\leq \rho \}$ and its
derivative does not vanish there. Then $f$ is locally
univalent in $\overline{\Delta_\rho(a)}$. Write
\begin{eqnarray}
\tilde{\delta}(w)=
\sup \{ r: \ \text{ $f$ is univalent in $\Delta_r(w)$} \},
\nonumber. \\
\delta(f,\rho,a)=
\min \{ \tilde{\delta}(w): \ w\in \Delta_\rho (a)\}.
\label{univalent}
\end{eqnarray}

\begin{remark}
%{\bf Remark}
The above definition of $\delta(f,\rho,a)$ is slightly
reminiscent of Bloch's constant $B$, cf. \cite{Mi}, but,
of course, they are very different.
\end{remark}

\begin{theorem}\label{commutant}
Let $A_1,A_2$ be $N\times N$ self-adjoint matrices, and
$\alpha_1,\dots,\alpha_N$ and $\beta_1,\dots,\beta_N$ be the
eigenvalues of $A_1$ and $ A_2$ respectively with
$
|\alpha_1|> |\alpha_2|> \dots > |\alpha_N|>0
$
%, \
and 
$
|\beta_1|> |\beta_2|> \dots > |\beta_N|>0.
$
Suppose that $\ell_1,\dots,\ell_N$ is a family of complex
lines of the form
$
\ell_j =
\{ \alpha_{n(j)}x+\beta_jy=1 \}, \ 1\leq j,n(j) \leq N
$
such that
\begin{itemize}
\item[(1)]
Each of the points
$(\frac{1}{\alpha_k},0), \ 1\leq k \leq N$
belongs to one of these lines; 

\item[(2)]
There exist $0<\rho <1$ and $0<\epsilon \ll \rho$
such that conditions (1) and (2) of Theorem~\ref{almost}
are satisfied by $\sigma_p(A_1, A_2)$ and each line
$\{ \alpha_{n(j)} x+\beta_jy=1\}$.
\end{itemize}
Then if $\epsilon $ is small enough, the norm of
the commutant
of $A_1$ and $A_2$ is at most of order $\epsilon^{1/2^N}$.
\end{theorem}

\begin{proof}
By Theorem~\ref{almost} the eigenvector $e_{n(1)}$ of $A_1$
that corresponds to eigenvalues $\alpha_{n(1)}$ is a
$
\sqrt{
\frac{8|\beta_1|(1+|\beta_1|^2)}{\rho-4|\beta_1|\epsilon}
\epsilon
}
$
-eigenvector of $A_2$, and relation \eqref{commutant1}
holds with $P_1$ being replaced with $P_{n(1)}$.  Write
$\epsilon_1=\epsilon$. We want to estimate $\epsilon_2$
such that the compression of $A_1^{(1)}$ and $A_2^{(1)}$ to
$span\{ e_k, \ k=1,\dots,N, \ k\neq n(1) \}$ satisfy
conditions (1) and (2) of the present Theorem with
$\epsilon_2, \ \rho/2$.

It follows from  \eqref{estimate} that in the eigenbasis
$e_1,...,e_N$ of the matrix $A_1$ every entry of the
$n(1)$-th row (and column) of the matrix $A_2$ except for
the one on the main diagonal  has absolute value that
does not exceed $C_1\sqrt{\epsilon_1}$, where
$
C_1 =
\sqrt{
\frac{8|\beta_1|(1+|\beta_1|^2)}{\rho- 4|\beta_1|\epsilon_1}
}.
$
Let
\begin{equation*}
{\mathcal P}(x,y) = \det \left [ xA_1+yA_2-I \right ], \
{\mathcal P}_1(x,y)= \det [ xA_1^{(1)}+yA_2^{(1)}-I ],
\end{equation*}
and let
\begin{equation}\label{minimum_derivative}
d =
\min
\Biggl\{
\biggl|
\alpha_{n(j)}\frac{\partial {\mathcal P}}{\partial x} +
\beta_j\frac{\partial {\mathcal  P}}{\partial y}
\biggr|
: \
\Bigl|
x- \frac{1}{\alpha_{n(j)}}
\Bigr|
\leq \epsilon, \
|y|\leq \epsilon, \ 1\leq j \leq N
\Biggr\} .
\end{equation}

Of course,  the determining polynomials ${\mathcal P}$
and ${\mathcal P} _1$ satisfy
\begin{equation}\label{polynomial_difference1}
{\mathcal P}(x,y)=
(\alpha_{n(1)}x+\beta_1y-1){\mathcal P}_1(x,y)+
{\mathcal Q}(x,y),
\end{equation}
where ${\mathcal Q}$ is a polynomial of degree $N$ whose
coefficients in absolute values do not exceed
$NC_1|\beta_1|^{N-1}\sqrt{\epsilon_1}$. Write
$$
M=\max \{ |1/\alpha_j|+\rho +1 \}.
$$
We obviously have
$$
|\Po (x,y)|\leq (|\alpha_{n(j_1)}+|\beta_1|+1)M|\Po _1(x,y)| +
NC_1|\beta_1|^{N-1}M^N\sqrt{\epsilon_1}.
$$
Now, let
$
(x,y)\in \sigma_p(A_1^{(1)}, A_2^{(1)})\cap
\{ |x-1/\alpha_{n(m)}|\leq \rho/2, \ |y|\leq \rho/2 \}
$
for some $2\leq m\leq N$.  Then
$|\Po(x,y)|\leq NC_1|\beta_1|^{N-1}M^N\sqrt{\epsilon_1}$.
Write $f(t)= \Po (x+t\alpha_{n(m)}, y+t\beta_m).$
Equation~\eqref{minimum_derivative} implies that
\begin{equation}\label{derivative}
|f^\prime (t)|\geq d>0
\end{equation}
for $|t|\leq \rho/2$, and, therefore, $f(t)$ is locally
univalent in the disk $\{ |t| \leq \rho/2\}.$ By
\eqref{univalent} $f$ is univalent  in the disk of radius
$
\tau=\min \{
\delta(F_{x,y}, \rho, 0): \ |x-1/\alpha_m|\leq \rho/2,
                         \ |y|\leq \rho/2
\},
$
where $F_{x,y}(w)=\Po( x +w\alpha_{n(m)}, y+w\beta_m)$,
so that this radius does not depend on the point $(x,y)$.
By \eqref{derivative} \  $|f^\prime (0)|\geq d$, so
Koebe's 1/4 theorem, cf. \cite[p.150]{K}, implies that
if $\epsilon_1$ is small enough so that
$
NC_1|\beta_1|^{N-1}M^N\sqrt{\epsilon_1}<\frac{d\tau}{4},
$
the function $f$ has a zero in
$\{ |t|\leq 4   NC_1|\beta_1|^{N-1}M^N\sqrt{\epsilon_1} \}$, and,
hence, the distance from
\[
\sigma_p(A_1^{(1)}, A_2^{(1)})\cap
\{ |x-1/\alpha_{n(m)}|\leq \rho/2, \ |y|\leq \rho/2 \}
\]
to $\sigma(A_1,A_2)$ does not exceed
$4NC_1|\beta_1|^{N-1}M^N\sqrt{\epsilon_1}$, and, therefore,
the distance from
\[
\sigma_p(A_1^{(1)}, A_2^{(1)})\cap
\{ |x-1/\alpha_m|\leq \rho/2, \ |y|\leq \rho/2 \}
\]
to the line $\{ \alpha_{n(m)} x+\beta_my=1\}$ does not exceed
$\epsilon_2=5NC_1|\beta_1|^{N-1}M^N\sqrt{\epsilon_1}$.

The fact that the eigenvalues of $A_2^{(1)}$ differ from
those of $A_2$ by a magnitude of order $\sqrt{\epsilon_1}$
follows directly from (\ref{distances}) and the fact that
for two compact normal operators the distance between their
spectra does not exceed the distance between them in the
operator norm topology, cf. \cite[Proposition 1]{GS}.

Fanally,  it follows from \eqref{polynomial_difference1}
that the difference between
$
\alpha_{n(k)}\frac{\partial \Po_1}{\partial x} +
\beta_k\frac{\partial \Po_1}{\partial y},
$
\  $k\neq 1$  and
$
\alpha_{n(k)}\frac{\partial \Po}{\partial x} +
\beta_k\frac{\partial \Po}{\partial y}
$
is of order of $\epsilon_2$, and, therefore, the
$(N-1)\times (N-1)$-dimensional matrices $A_1^{(1)}$ and
$A_2^{(1)}$ satisfy the conditions of this theorem with
$\epsilon_2$ and $\rho/2$.
Continuing inductively we arrive at the claimed estimate.
\end{proof}

\section{Concluding remarks}
\label{S:concluding-remarks}

Here we want to state and discuss four problems
immediately related to the material of the preceding
sections. The first one was essentially stated in the
Introduction.
It was mentioned there that, according to \cite{SYZ},  if
$A_1,\dots, A_n$ are compact, then on every compact subset of
$\C^n$ the joint spectrum $\sigma_p(A_1,\dots, A_n)$ has a
global defining function. Our first problem is:

%\vspace{.2cm}

%{\bf Problem 1}.
%{\it
\begin{problem}
Describe those globally defined in $\C^n$ analytic sets
that have spectral representation. In particular, which
real globally defined analytic sets (that is, zeros of
entire functions with real Taylor coefficients) have
self-adjoint spectral representation?
\end{problem}

In the case of matrices the similar problem for general matrices was solved 
by Dickson \cite{D} and for self-adjoint matrices by Helton and Vinnikov 
\cite{HV}. 

\vspace{.2cm}

%\vspace{.2cm}

The second problem is related to the example in Section~6
and Theorem~B. It was required in Theorem \ref{circle} that
%$|\beta|=
$\Vert A_2\Vert=1$ and this condition substantially
differentiates this result from Theorem~B. We
wonder whether this condition is redundant and the result
holds without it. More generally, it is natural to ask
whether a complete analog of Theorem~B for algebraic curves
of order higher than $1$ is valid. Thus, we are compelled
to pose the following problem.

%\vspace{.2cm}

%{\bf Problem 2}
%{\it

\begin{problem}
Suppose that $A$ and $B$ are self-adjoint, \ $A$ is
invertible and $\sigma_p(A,B)$ contains a real
algebraic curve $\Gamma$ of order $k$ that meets the
$x$- and $y$-axes at points
$(\alpha_1,0),\dots, (\alpha_k,0)$
and $(\beta_1,0),\dots, (\beta_k,0)$ respectively.
Also suppose
that all these points of intersection of $\Gamma$ with the
coordinate axes are isolated spectral points of the
corresponding operators. If $\sigma_p(A^{-1},B)$
contains an algebraic curve $\Gamma_1$ of the same order
$k$ that meets the coordinate axes at points
$(1/\alpha_1,0),\dots,(1/\alpha_k,0)$ and
$(\beta_1,0),\dots,(\beta_k,0)$ respectively, does this imply
that $A$ and $B$ have a common $k$-dimensional reducing
subspace?
\end{problem}

%\vspace{.2cm}

We saw that the appearance of circular arcs in the joint
spectrum of $A$ and $B$ under the assumptions of
Theorem~\ref{circle} implies that the group generated
by $A$ and $B$ ``contains a representation'' of
Coxeter's group $BC_2$ in the sense that the the pair
$(A,B)$ is decomposable having a finite dimensional
reducing subspace and the restriction of the group to
this subspace represents $BC_2$. It would be interesting
to find out whether there are other curves or surfaces
in joint spectra of operator tuples that are associated
with other symmetry groups. For instance,
Theorem~\ref{line_in_spectrum} shows that this is true
for lines appearing in joint spectra of $A$ and $B$,
when the corresponding lines appear in
$\sigma_p(A^{-1}, B)$. In this case the represented
group is a product of copies of the Coxeter group $A_1$.

%\vspace{.2cm}

%{\bf Problem 3}
%{\it

\begin{problem}
Which curves (or surfaces) appearing in the joint spectrum
of a pair of operators (tuple of operators) indicate that
the group generated by this pair (tuple)
"contains a representation" of a corresponding symmetry
group?
\end{problem}

%\vspace{.2cm}

The last problem we would like to mention is related to
the norm estimate of the commutant of two matrices, or
more generally two compact operators.
The estimate given by Theorem~\ref{commutant} seems to be
rather rough. Besides, this estimate does not allow any
extension of the result of Theorem~\ref{commutant} to an
infinite dimensional case. One possible way of improving
the estimate is to consider that the proper spectrum is
close to a set of lines not locally, but on a big compact
subset of $\C^2$. Alternatively, we might impose the
condition that the joint projective spectrum in 
$\C\mathbb P^2$
is close to s set of projective lines  in Fubini study
metric. This leads us to the following problem.

%\vspace{.2cm}

%{\bf Problem 4}.
%{\it

\begin{problem}
Let $A$ and $B$ be self-adjoint compact operators
acting on a separable Hilbert space $H$. Suppose that
the distance  from
$\sigma(A,B,I)$ to a set of projective lines that
contains $\{ [x:y:0] \}$ and satisfies the following
condition: 
\begin{itemize}
\item[(C)] 
the intersection of this set of lines with the lines 
$\{ [0:y:z]\}$ and $\{ [x:0:z]\}$ coincides with the
inverse spectra
$\sigma(A)^{-1}$ and
$\sigma(B)^{-1}$,
counting multiplicities;  
\end{itemize}
does not exceed $\epsilon$ 
in the  Fubini study metric.
%\end{itemize}
Estimate the norm of the commutant $[A,B]$ in terms
of $\epsilon$.
\end{problem}

\end{document}